\documentclass[12pt,a4paper, fleqn]{article}
\usepackage{etex}
\reserveinserts{28}

\usepackage[hmargin=25mm,top=25mm,bottom=30mm]{geometry}  

\usepackage{amsmath}
\usepackage{hyperref}
\hypersetup{colorlinks=false, linkbordercolor={1 1 1}, citebordercolor={1 1 1}}
\usepackage{hypernat}
\usepackage{graphicx}
\usepackage{wrapfig}

\usepackage{pstricks}
\usepackage{pstricks-add}
\usepackage{multido}
\usepackage{pst-fill}
\usepackage{pifont}

\usepackage{pdfsync}
\usepackage{latexsym}
\usepackage{amssymb}
\usepackage{amsthm}
\usepackage[small,bf]{caption2}
\usepackage{times}
\usepackage{color}

\definecolor{mygray}{gray}{0.5}

\numberwithin{equation}{section}
\newtheorem{theorem}{Theorem}[section]
\newtheorem{lemma}[theorem]{Lemma}
\newtheorem{corollary}[theorem]{Corollary}

\newtheorem{proposition}[theorem]{Proposition}
\theoremstyle{definition}

\def\s{\sigma}

\def\Zd{{\Z^d}}

\def\({{\Bigl(}}
\def\){{\Bigr)}}

\def\one{{\mathbf 1}}

\def\square{\ifmmode\sqr\else{$\sqr$}\fi}
\def\sqr{\vcenter{
         \hrule height.1mm
         \hbox{\vrule width.1mm height2.2mm\kern2.18mm\vrule width.1mm}
         \hrule height.1mm}}                  
\def\qed{\hfill \square}

\def\R{{\mathbb R}}
\def\Z{{\mathbb Z}}
\def\N{{\mathbb N}}

\def\cL{{\mathcal{L}}}

\def\cG{{\mathcal{G}}}
\def\cI{{\mathcal{I}}}

\def\cN{{\mathcal{N}}}

\def\bA{{\mathbf{A}}}

\def\bD{{\mathbf{D}}}

\def\bK{{\mathbf{K}}}

\def\bA{{\mathcal{A}}}

\def\bD{{\mathcal{D}}}

\def\bK{{\mathcal{K}}}

\def\rI{{\mathrm{I}}}
\def\uU{{V}}

\def\one{{\mathbf 1}}
\def\s{\sigma}

\def\a{{\alpha}}

\title{Finite cycle Gibbs measures on permutations of $\Zd$}
\author{In\'es Armend\'ariz\thanks{Departamento de Matem\'atica, Universidad de Buenos Aires, Argentina. iarmend@dm.uba.ar},
Pablo A. Ferrari\thanks{Departamento de Matem\'atica, Universidad de Buenos Aires, Argentina, and Instituto de Matem\'atica e Estat\'istica, Universidade de S\~ao Paulo, Brazil. pferrari@dm.uba.ar},
Pablo Groisman\thanks{Departamento de Matem\'atica, Universidad de Buenos Aires and IMAS-CONICET, Buenos Aires, Argentina. pgroisma@dm.uba.ar}, 
Florencia Leonardi\thanks{Instituto de Matem\'atica e Estat\'istica, Universidade de S\~ao Paulo, Brazil. leonardi@ime.usp.br}
}

\date{}

\begin{document}

\sloppy
\maketitle

\begin{abstract} 
\noindent We consider Gibbs distributions on the set of permutations
of $\Z^d$ associated to the 
Hamiltonian $H(\sigma):=\sum_{x} \uU(\sigma(x)-x)$, 
where $\sigma$ is a permutation and $\uU:\Z^d\to\R$ is a strictly convex potential.
Call finite-cycle those permutations composed by finite cycles only. We give conditions on $\uU$ ensuring that for large enough temperature $\alpha>0$ 
there exists a unique infinite volume ergodic Gibbs measure
$\mu^\alpha$ concentrating mass on finite-cycle permutations; this measure is equal to the thermodynamic limit of
the specifications with identity boundary conditions.  We construct
$\mu^{\alpha}$ as the unique invariant measure of a Markov process on
the set of finite-cycle permutations that can be seen as a loss-network, a continuous-time birth and
death process of cycles interacting by exclusion, an approach proposed
by Fern\'andez, Ferrari and Garcia. Define $\tau_v$ as the shift permutation
$\tau_v(x)=x+v$. In the Gaussian case $\uU=\|\cdot\|^2$, we show that for each $v\in\Z^d$,
$\mu^\alpha_v$ given by $\mu^\alpha_v(f)=\mu^\alpha[f(\tau_v\cdot)]$ is an ergodic Gibbs measure
equal to the thermodynamic limit of the specifications with
$\tau_v$ boundary conditions. For a general potential $\uU$, we prove the existence of Gibbs measures
$\mu^\alpha_v$ when $\alpha$ is bigger than some $v$-dependent value.
\end{abstract}

\section{Introduction}

The Feynman-Kac representation of the Bose gas consists of trajectories of
interacting Brownian motions in a fixed time interval, which start and
finish at the points of a spatial point process \cite{feynman} . In order to attempt a rigorous analysis of the 
model,  several simplifications have been proposed over the years
\cite{feynman, fichtner, kikuchi,  MR0127331}. In the resulting model, the
starting and ending points belong to the
$d$-dimensional lattice, and the interaction is reduced to an exclusion condition on the
paths at the beginning and the end of the time interval. The state space is therefore the
set of permutations or bijections $\sigma:\Z^d\to\Z^d$. 

For a finite set
$\Lambda\subset\Z^d$, denote by $S_\Lambda$ the set of permutations $\sigma$ that reduce to 
the identity outside $\Lambda$, i.e., 
\begin{align}
  S_\Lambda:=
\{\sigma \in S\,:\,\sigma(x)=x \hbox{ if } x\notin\Lambda\}.
\end{align}
A function $\uU:\Z^d\to\R^+ \cup \{+\infty\}$ such that $\uU(\vec0)=0$ is called a \emph{potential}. We assume $\uU$ is strictly convex and define the Hamiltonian
\begin{equation}
  \label{gh}
 H_\Lambda(\sigma):=\sum_{x\in\Lambda}V(\sigma(x)-x), \quad\sigma\in S_\Lambda, 
\end{equation}
and  associated measure $G_\Lambda$,
\begin{equation}
  \label{mulambda}
  G_\Lambda(\sigma) := \frac1{Z_\Lambda} e^{-\alpha H_\Lambda(\sigma)},
\end{equation}
where $Z_\Lambda$ is a normalizing constant. The nonnegative parameter $\alpha$ is called the temperature; 
we omit
the dependence of $G_{\Lambda}$ on $\alpha$. We refer to the condition
$\sigma(x)=x$ if $x\notin\Lambda$ as an \emph{identity boundary
  condition}, and the finite volume measure $G_\Lambda$ associated to
a finite set $\Lambda\subset\Z^d$ is called a \emph{specification}. 

When the potential is Gaussian, $V(x)=\|x\|^2$, the value $e^{-\alpha\|\sigma(x)-x\|^2}$ is proportional to the
density at the site $\sigma(x)$ of a Gaussian distribution with mean
$x$ and variance $1/(2\alpha)$. Hence, $G_\Lambda$ is proportional to the joint
density of the arrival points at time $1/(2\alpha)$ of a family of 
independent Brownian motions started at each 
point in $\Lambda$, which are conditioned to arrive at distinct points
of $\Lambda$ at that time. This is the case arising from the Feynmann-Kac representation of the Bose gas.

Given permutations $\tau,\sigma$, define the composed permutation $\tau\sigma$
by  $(\tau\sigma)(x):=\tau(\sigma(x))$ and let $\mu\tau$ be the law of
$\tau\sigma$ when $\sigma$ is distributed according to $\mu$, that is 
\begin{equation}
\label{comp}
(\mu\tau)f=\int\mu(d\sigma) f(\tau\sigma), 
\end{equation}
for continuous real functions $f$.
For any vector
$v\in\Z^d$ denote by $\tau_v$ the shift permutation given by
\begin{equation}
  \label{xv1}
  \tau_v(x):=x+v\,.
\end{equation}
A permutation $\tau$ is called a \emph{ground state} if $\tau$ is a local minimum of the
Hamiltonian $H_{\Z^d}$. Since $\uU$ is strictly convex, the shift permutation $\tau_v$ is a ground state for any vector $v\in \Z^d$.  


\paragraph{Results} Our main results are the following. 

 {\it Identity boundary conditions. } In Theorem \ref{t1} we define a function $\alpha^*(\uU)$ such that when it is finite,
 for any $\alpha>\alpha^*(\uU)$,
there exists an ergodic Gibbs measure $\mu$ equal to the thermodynamic limit of the specifications with identity boundary conditions at temperature $\alpha$. The measure $\mu$  concentrates on finite-cycle permutations. 

 {\it Shift boundary conditions. } In Theorem \ref{t2} we fix $v\in\Z^d$ and extend the results of Theorem \ref{t1} to $\tau_v$-boundary conditions. That is, we define $\alpha^*_v(V)$, and assuming that it is finite, we show that for any temperature $\alpha>\alpha^*_v(V)$ there exists an ergodic  Gibbs measure $\mu_v$ associated to $\tau_v$ boundary 
 conditions such that 
 $\mu_v\tau_{-v}$ concentrates on finite-cycle permutations. 
 
 {\it Gaussian potential. }The physically relevant Gaussian potential $V(x)=\|x\|^2$ is covered by Theorem \ref{t1}; in this case the results for $\tau_v$-boundary conditions 
 follow directly from the observation that the specifications $G_{\Lambda|\tau_v}$ matching $\tau_v$ at the boundary satisfy $G_{\Lambda|\tau_v}=G_{\Lambda}\tau_{v}$, a relation that extends to the limit $\mu_v=\mu \tau_{v}$. In particular, here $\alpha_v^*(V)$ is the same for all $v\in \Z^d$, $\alpha^*_v(V)=\alpha^*(V)$.
 
The statements of these results establish the existence of Gibbs measures $\mu$ as a weak limit of specifications. In fact, we obtain pointwise limits. For instance, in the proof of Theorem \ref{t1} we construct a coupled family of permutations $(\zeta_\Lambda,\,\Lambda\subseteq \Z^d)$, each $\zeta_{\Lambda}$ distributed according to $G_{\Lambda}$ ($G_{\Z^d}=\mu$), 
 such that  for $x\in \Z^d$ the random variables $\zeta_\Lambda(x)$ converge almost surely  to $\zeta_{\Z^d}(x)$, as $\Lambda\nearrow\Z^d$.  

In Section \ref{s5} we compute bounds for $\alpha^*(V)$. In the Gaussian case these computations yield explicit bounds, see \eqref{explicit}.

\paragraph{Approach}
The proofs follow the approach of Fern\'andez, Ferrari and Garcia
\cite{MR1849182}, relying on the fact that the Peierls-contour representation of the low
temperature Gibbs measure for the Ising model is reversible for a loss network of contours interacting by exclusion. In the case of identity boundary conditions, instead
of contours, we consider the finite cycles that compose a permutation.
 Let $\Gamma$ be the set of finite cycles on $\Z^d$ with length larger than 1.  A finite-cycle permutation is represented as a ``gas'' of finite cycles in $\Gamma$, and the Gibbs measure can be described as a product of
independent Poisson random variables in the space $\{0,1,\dots\}^\Gamma$, conditioned to non overlapping of cycles, that is, each site $x\in\Z^d$ belongs to at most one cycle. This is automatically well defined in finite volume. We explicitly construct an infinite volume random configuration $\eta\in\{0,1\}^\Gamma$ with non overlapping cycles, $\eta(\gamma)=1$ means that the cycle $\gamma$ is present in the configuration $\eta$. This configuration is naturally associated to the permutation $\sigma$ composed by the cycles indicated by $\eta$. We then show that $\sigma$ is the almost sure limit as $\Lambda\nearrow\infty$ of permutations in $S_\Lambda$ distributed according to the specifications $G_\Lambda$.

The  loss network is a continuous-time Markov process $\eta_t\in\{0,1\}^\Gamma$, having as unique invariant measure the target Gibbs
measure. In this process, each cycle $\gamma$ attempts to appear independently at a rate $w(\gamma)$ defined later in \eqref{weight}, and $\gamma$ is allowed to join the existing configuration only if it does not overlap with the already present cycles. Cycles also die, independently, at rate 1. If $\a$ is sufficiently large this process is well defined in infinite volume, and a realization of the stationary process running for all  $t\in \R$ can be constructed as a function of a family of space-time Poisson processes, the usually called Harris graphical construction. The condition for the existence of the stationary process is related to the absence of oriented percolation of cycles in the space--time realization of a \emph{free process} in $\{0,1,\dots\}^\Gamma$, where all cycles
are allowed to be born, regardless whether they overlap with pre-existing cycles or not. The no-percolation
condition follows from dominating the percolation cluster by a
subcritical multitype branching process, a standard technique, see for instance
\cite{MR1707339}. The subcriticality condition for the branching process leads
to the condition $\alpha>\alpha^*$.

\paragraph{Background and further prospects}  The existence of a Gibbs measure concentrating on
finite-cycle permutations of $\Z^d$ was first obtained by Gandolfo, Ruiz and
Ueltschi \cite{MR2360227} in the large temperature regime for the Gaussian potential.
Recently, Betz \cite{Betz} gave a condition yielding tightness of the
specifications for a more general family of potentials, for any value of $\alpha$, 
his results imply that thermodynamic limits of specifications with
identity boundary conditions exist for any $\a>0$ and dimension
$d$. However, the problem of identifying these limits and their typical cycle
length remains open.

Biskup and Richthammer \cite{BR} consider the one dimensional case and strictly convex potentials satisfying some additional growth conditions. They prove that the set of all 
ground states associated to $H$ in \eqref{gh} is $\{\tau_v:v\in\Z\}$, $\tau_v(x)=x+v$ as in \eqref{xv1}, and 
that for each ground state $\tau_v$ and temperature $\alpha>0$ there is a Gibbs measure $\mu^{\a}_v$. Furthermore, they show that the set of extremal Gibbs measures is $\cG_{\alpha,e}=\{\mu^{\a}_v,\,
v\in\Z\}$, that is, each extremal Gibbs measure is associated to a ground
state. The measure $\mu^{\a}_v$ is translation invariant and supported on configurations having exactly $n$ infinite cycles. 
They also prove that for
any $\a>0$, the measure $\mu_v^\a$ has a regeneration property, which in the case $v=0$ entails the convergence as $\Lambda\nearrow\Z$ of the specifications $G_\Lambda$ with identity boundary conditions to
$\mu_0^{\a}$. In particular, this implies that for $d=1$, identity boundary
conditions lead to finite cycles, for all temperatures.

\emph{Infinite cycles. }
 In $d$--dimensions our results say that under identity boundary conditions, for
 $\alpha$ large enough,  the Gibbs measures concentrate on finite-cycle permutations. On the other hand, for the Gaussian potential and  small $\a$, Gandolfo, Ruiz and Ueltschi  \cite{MR2360227} performed numerical simulations of the 3-dimension specification
associated to a box $\Lambda$ yielding cycles with \emph{macroscopic} length, i.e., length that grows proportionally to the size of $\Lambda$.
More recent numerical results by Grosskinsky, Lovisolo and Ueltschi \cite{MR2903040} suggest that the scaled down size of these 
macroscopic cycles converges to a Poisson-Dirichlet distribution.  See also Goldschmidt, Ueltschi and Windridge
\cite{MR2868048} for a discussion relating cycle representations and fragmentation-coagulation models, where the Poisson-Dirichlet distributions appear naturally. The authors in \cite{MR2903040}
argue that the situation should be similar in higher dimensions, in contrast to
the case $d=2$. In 2--dimensions it is expected that the size of the 
cycles grows as $\Lambda\nearrow \Z^2$, but in this case the length would
not be macroscopic, a  conjecture that is supported by numerical simulations in
\cite{Betz, MR2360227}. The question remains whether a positive fraction
of sites belongs to these mesoscopic cycles. Betz  \cite{Betz} provides numerical evidence that for $d=2$ long cycles are fractals in the thermodynamic limit, and conjectures a connection to Schramm-Loewner evolution.

\begin{figure}[t]
\begin{center}
\psset{unit=0.8cm}
\begin{pspicture}(15,8)(0,0)  
\psset{linewidth=1.2pt,arrowsize=7pt,arrowlength=0.7, arrowinset=0.2}
\psset{linecolor=mygray}
\rput{0}(1,2){
\rput{0}(-0.28,-0.28){\psline(0,0)(5.5,0)(5.5,5.5)(0,5.5)(0,0)}
\multido{\nx=-0.5+0.5}{13}{
	\multido{\ny=-0.5+0.5}{13}{
		\pscircle[fillstyle=solid,fillcolor=black,linecolor=black](\nx,\ny){0.043}}}
\rput{0}(2.5,-1.3){ground state $\xi$}
\multido{\nx=-0.5+0.5,\ny=0.0+0.5}{12}{\psline{->}(\nx,2.5)(\ny,2.5)}
}
\rput{0}(9,2)
{
\rput{0}(-0.28,-0.28){\psline(0,0)(5.5,0)(5.5,5.5)(0,5.5)(0,0)}
\multido{\nx=-0.5+0.5}{13}{
	\multido{\ny=-0.5+0.5}{13}{
		\pscircle[fillstyle=solid,fillcolor=black,linecolor=black](\nx,\ny){0.043}}}
\rput{0}(2.5,-1.3){permutation $\sigma\colon G_{\Lambda|\xi}(\sigma)>0$}
\psline{->}(-0.5,2.5)(0,2.5)
\psline{->}(0,2.5)(0.5,3.5)
\pscurve{->}(0.5,3.5)(1,3.7)(1.5,3.5)
\psline{->}(1.5,3.5)(2,4)
\psline{->}(2,4)(2.5,3.5)
\pscurve{->}(2.5,3.5)(3,3.7)(3.5,3.5)
\pscurve{->}(3.5,3.5)(3.2,3)(2.8,2.5)(2.5,1.5)
\pscurve{->}(2.5,1.5)(3,1.7)(3.5,1.5)
\psline{->}(3.5,1.5)(4,2)
\psline{->}(4,2)(4.5,3)
\psline{->}(4.5,3)(5.5,2.5)
\psline{->}(0.5,0.5)(1,1.5)
\pscurve{->}(1,1.5)(1.25,1.7)(1.5,1.5)
\psline{->}(1.5,1.5)(2,0.5)
\pscurve{->}(2,0.5)(1.25,0.25)(0.5,0.5)
\pscurve{->}(1,4.5)(0.75,3.5)(1,2.5)
\psline{->}(1,2.5)(2.5,3)
\pscurve{->}(2.5,3)(2.75,3.75)(2.5,4.5)
\pscurve{->}(2.5,4.5)(1.75,4.7)(1,4.5)
\psline{->}(4,4.5)(4.5,3.5)
\psline{->}(4.5,3.5)(5,5)
\psline{->}(5,5)(4,4.5)
}
\end{pspicture}
\end{center}
\caption{A dot at site $x$ means that $\xi(x)=x$ while an arrow from $x$ to $y$
  means that $\xi(x)=y$. The left picture represents the ground state $\xi$  described
  in \eqref{xiy}.  On the right we see a permutation with positive probability according to
  $G_{\Lambda|\xi}$ defined in \eqref{specif}. The square represents the box $\Lambda$.}
 \label{f1}
\end{figure}
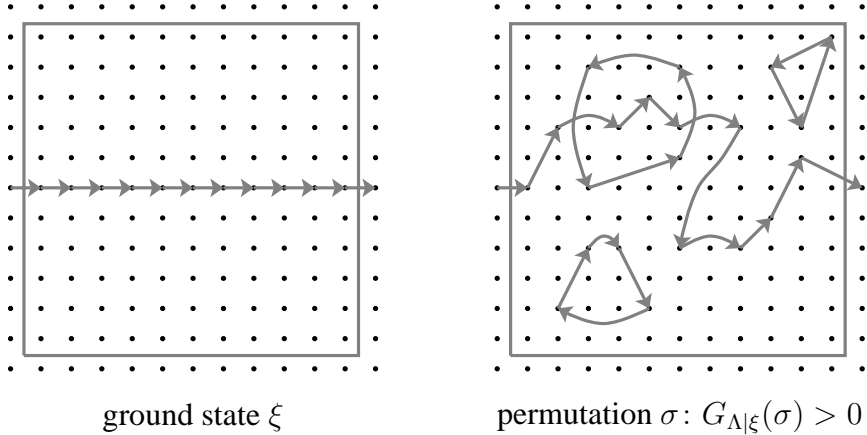

\emph{Domain of attraction of Gibbs measures. } Let $x=(x_1,\dots,x_d) \in \Z^d$ and $e_1=(1,0,\dots,0)$ denote the first vector in the
canonical basis.
In a forthcoming paper, Yuhjtman considers the Gaussian potential with ground state $\xi$ defined by
\begin{align}
\label{xiy}
\xi(x)=\left\{\begin{array}{ll}
x+e_1 & \hbox{if }x_2=\dots=x_d= 0,\\
x & \hbox{ otherwise},
\end{array}\right. 
\end{align}
(see Figure~\ref{f1}) and shows that  the thermodynamic limit of $G_{\Lambda|\xi}$, the specifications with boundary conditions given by $\xi$, is equal to the measure $\mu$ associated to identity boundary conditions. In particular, in dimensions higher than 1, the one-to-one correspondence between ground states and extremal Gibbs measures fails to hold.
 It would be interesting to find the domain of attraction of each Gibbs
state. That is, if $\mu$ is a Gibbs measure, one would like to characterize the set $\big\{\xi \,:\, \lim_{\Lambda\nearrow \Z^d}G_{\Lambda|\xi}=\mu\big\}$.

\begin{figure}[t]
\begin{center}
\psset{unit=0.8cm}
\psset{linewidth=1.2pt,arrowsize=7pt,arrowlength=0.7, arrowinset=0.2}
\begin{pspicture}(15,8)(0,0)  
\psset{linecolor=mygray}
\rput{0}(1,2){
\rput{0}(-0.28,-0.28){\psline(0,0)(5.5,0)(5.5,5.5)(0,5.5)(0,0)}
\multido{\nx=-0.5+0.5}{13}{
	\multido{\ny=-0.5+0.5}{13}{
		\pscircle[fillstyle=solid,linecolor=black,fillcolor=black](\nx,\ny){0.043}}}
\rput{0}(2.5,-1.3){ground state $\xi$}
\multido{\nz=2.0+-0.5}{6}{
	\multido{\nx=-0.5+0.5,\ny=0.0+0.5}{12}{
		\psline{->}(\nx,\nz)(\ny,\nz)}}	
}
\rput{0}(9,2)
{
\rput{0}(-0.28,-0.28){\psline(0,0)(5.5,0)(5.5,5.5)(0,5.5)(0,0)}
\multido{\nx=-0.5+0.5}{13}{
	\multido{\ny=-0.5+0.5}{13}{
		\pscircle[fillstyle=solid,fillcolor=black,linecolor=black](\nx,\ny){0.043}}}
\rput{0}(2.5,-1.3){permutation $\sigma\colon G_{\Lambda|\xi}(\sigma)>0$}

\multido{\nx=-0.5+0.5,\ny=0.0+0.5}{12}{
		\psline{->}(\nx,-0.5)(\ny,-0.5)}	
\psline{->}(-0.5,1.5)(0.0,1.5)
\psline{->}(0.0,1.5)(1.5,2.5)
\pscurve{->}(1.5,2.5)(2,2.7)(2.5,2.5)
\psline{->}(2.5,2.5)(3,3)
\pscurve{->}(3,3)(3.1,2.9)(3.5,2.3)(4,1)
\pscurve{->}(4,1)(4.5,1.75)(5,2.2)(5.5,2)
\psline{->}(-0.5,1)(0,1)
\pscurve{->}(0,1)(0.2,1.4)(0.4,2.7)(1,3)
\pscurve{->}(1,3)(1.5,3.2)(2,3)
\pscurve{->}(2,3)(2.5,3.75)(3,4)
\pscurve{->}(3,4)(3.75,4.2)(4.5,4)
\pscurve{->}(4.5,4)(4.7,3.25)(4.5,2.5)
\pscurve{->}(4.5,2.5)(4.3,1.5)(5,1)
\psline{->}(5,1)(5.5,1)
\psline{->}(-0.5,2)(0.0,2)
\pscurve{->}(0.0,2)(0.5,1.7)(1,1.2)(2,0.5)
\psline{->}(2,0.5)(3,1)
\psline{->}(3,1)(3.5,0.0)
\pscurve{->}(3.5,0.0)(4.25,0.25)(5,0.0)
\psline{->}(5,0.0)(5.5,0.0)
\psline{->}(-0.5,0.5)(0.0,0.5)
\psline{->}(0.0,0.5)(2.5,0.0)
\pscurve{->}(2.5,0.0)(3,0.25)(3.5,1)
\psline{->}(3.5,1)(4.5,0.5)
\pscurve{->}(4.5,0.5)(5,0.75)(5.5,0.5)
\psline{->}(-0.5,0)(0,0)
\pscurve{->}(0,0)(0.75,1)(1.5,1.5)
\pscurve{->}(1.5,1.5)(2.5,1.75)(3.5,1.5)
\psline{->}(3.5,1.5)(4,2)
\psline{->}(4,2)(5.5,1.5)
\psline{->}(4,4.5)(3.5,3.5)
\pscurve{->}(3.5,3.5)(4.7,3.4)(4.75,4.25)(5,5)
\psline{->}(5,5)(4,4.5)
\rput{0}(-0.5,3){
\psline{->}(0.5,0.5)(1,1.5)
\pscurve{->}(1,1.5)(1.25,1.7)(1.5,1.5)
\psline{->}(1.5,1.5)(2,1)
\psline{->}(2,1)(0.5,0.5)}
}
\end{pspicture}
\begin{pspicture}(15,8)(0,0)  
\psset{linecolor=mygray}
\rput{0}(1,2){
\rput{0}(-0.28,-0.28){\psline(0,0)(5.5,0)(5.5,5.5)(0,5.5)(0,0)}
\multido{\nx=-0.5+0.5}{13}{
	\multido{\ny=-0.5+0.5}{13}{
		\pscircle[fillstyle=solid,fillcolor=black,linecolor=black](\nx,\ny){0.043}}}
\rput{0}(2.5,-1.3){ground state $\xi'$}
\multido{\nz=0.0+1.0}{6}{
	\multido{\nx=-0.5+0.5,\ny=0.0+0.5}{12}{
		\psline{->}(\nx,\nz)(\ny,\nz)}}	
}
\rput{0}(9,2)
{
\rput{0}(-0.28,-0.28){\psline(0,0)(5.5,0)(5.5,5.5)(0,5.5)(0,0)}
\multido{\nx=-0.5+0.5}{13}{
	\multido{\ny=-0.5+0.5}{13}{
		\pscircle[fillstyle=solid,fillcolor=black,linecolor=black](\nx,\ny){0.043}}}
\rput{0}(2.5,-1.3){permutation $\sigma\colon G_{\Lambda|\xi'}(\sigma)>0$}
\psline{->}(-0.5,2)(0,2)
\psline{->}(0.0,2)(1.5,2.5)
\pscurve{->}(1.5,2.5)(2.25,2.7)(3,2.5)
\pscurve{->}(3,2.5)(3.5,2.25)(4,1)
\pscurve{->}(4,1)(4.75,1.2)(5.5,1)
\psline{->}(-0.5,1)(0.0,1)
\pscurve{->}(0.0,1)(0.2,1.4)(0.4,2.7)(1,3)
\pscurve{->}(1,3)(1.5,3.2)(2,3)
\psline{->}(2,3)(2.5,4)
\pscurve{->}(2.5,4)(3.5,4.2)(4.5,4)
\psline{->}(4.5,4)(5,2)
\psline{->}(5,2)(5.5,2)
\psline{->}(-0.5,0.0)(0,0)
\pscurve{->}(0,0)(0.75,1)(1.5,1.5)
\pscurve{->}(1.5,1.5)(2.5,1.75)(3.5,1.5)
\psline{->}(3.5,1.5)(4,2.5)
\pscurve{->}(4,2.5)(4.5,2.75)(5,4)
\psline{->}(5,4)(5.5,4)
\psline{->}(-0.5,3)(0,3)
\pscurve{->}(0,3)(0.5,2.75)(1,2)
\psline{->}(1,2)(0.5,1.5)
\psline{->}(0.5,1.5)(2,0.5)
\pscurve{->}(2,0.5)(2.75,0.25)(3.5,0.5)
\pscurve{->}(2,0.5)(2.75,0.25)(3.5,0.5)
\psline{->}(3.5,0.5)(5,0.0)
\psline{->}(5,0.0)(5.5,0.0)
\psline{->}(-0.5,4)(0,4)
\pscurve{->}(0,4)(0.75,4.25)(1.5,4)
\psline{->}(1.5,4)(3,5)
\pscurve{->}(3,5)(3.75,4.75)(4.5,5)
\pscurve{->}(4.5,5)(5,4.75)(5.5,5)
\psline{->}(-0.5,5)(0,5)
\psline{->}(0,5)(1,4.5)
\pscurve{->}(1,4.5)(2.5,4.7)(3.5,3.5)
\psline{->}(3.5,3.5)(5.5,3)
\psline{->}(1,0.5)(2.5,1)
\pscurve{->}(2.5,1)(2,0.2)(1,0.5)
}
\end{pspicture}
\end{center}
\caption{The ground states $\xi$ and $\xi'$ and permutations with positive probability for the specifications with $\xi$ and $\xi'$ boundary conditions respectively.}
\label{f2}
\end{figure}
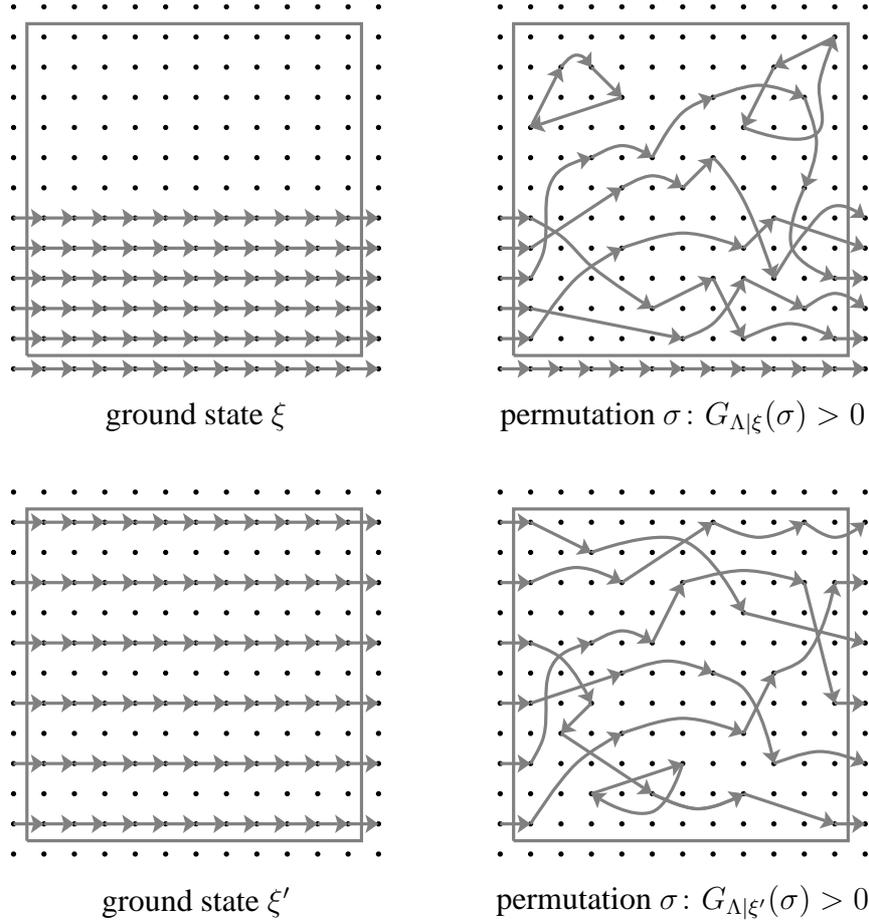

\emph{Further translation invariant Gibbs measures. } 
Set $d\ge 2$ and consider the ground states $\xi, \xi':\Z^d\to \Z^d$ given by
\begin{align}
\label{xiy2}
{\xi}(x)=\left\{\begin{array}{ll}
x &\hbox{if } x_d\geq 0,\\
x+e_1 &\hbox{if } x_d <0
\end{array}\right. ,
\qquad
{\xi'}(x)=\left\{\begin{array}{ll}
x &\hbox{if } x_d\hbox{ is even},\\
x+e_1& \hbox{if } x_d\hbox{ is odd}.
\end{array}\right. 
\end{align}

Our approach requires translation invariance of the
boundary conditions, which are satisfied neither by $\xi$ nor by $\xi'$ (see Figure~\ref{f2}).
The conjecture is that the thermodynamic
limit arising from any of these boundary conditions should lead to a Gibbs measure
with $\frac12$ - density of paths crossing the hyperplane $x_1=0$ from left to right.
In connection to these ground states, it would be interesting to describe the macroscopic shape determined by these
left-right crossing paths.

\emph{Permutations of point processes. } When the points are distributed according to a point process there are two possibilities. In the quenched case one studies the random permutation of a fixed point configuration. In this case we expect that our approach would be useful to show that for almost all point configuration there is a unique Gibbs measure when the temperature is high enough in relation to the point density $\rho$.
 The 1--dimensional quenched case is studied by Biskup and Richthammer 
\cite{BR}, who prove that there are no infinite cycles for any value of the temperature. S\"uto \cite{MR1241339, MR1945163} investigates the annealed case, 
where one jointly averages point positions and permutations. By integrating over the former, it is then possible to explicitly identify the temperature $\alpha_0$ below which 
infinite cycles appear, S\"uto points out that this is equivalent to Bose-Einstein condensation in the Bose gas. These results are generalized by Betz and Ueltschi in \cite{MR2461985}.

\paragraph{Organization of the article}  We introduce notation and describe rigorously the
results in Section \ref{s2}, we then sketch the techniques in Subsection
\ref{s2.1}. We construct the loss network approach of \cite{MR1849182} in Section
\ref{s3}, and prove the main results in Section \ref{s4}. Section \ref{s5} contains bounds for $\alpha^*(V)$. 

\section{Notation and Results}
\label{s2}

Denote by $S$  the set of permutations of $\Zd$, that is
\[
S:=\{\sigma:\Zd \to \Zd, \sigma \, \mbox{bijective} \}\,,
\]
equipped  with the product topology generated by the sets $\{\sigma\in S: \sigma(x)=y\}$, $x,y\in\Z^d$, and the associated Borel sigma-algebra $\mathcal B$. 
Given a permutation $\xi \in S$ and a finite set $\Lambda\subseteq\Zd$, let
\begin{equation}
\label{restrict}
S_{\Lambda|\xi}:=\{\sigma \in S : \sigma(x)=\xi(x), \,x \in  \Lambda^c\},
\end{equation}
be the set of permutations that match $\xi$ outside of $\Lambda$. Let $\rI$ be the identity permutation,  $\rI(x)=x$ for all $x\in\Zd$, and denote $S_\Lambda:=S_{\Lambda|\rI}$. Let
$\uU:\Z^d\to\R^+\cup \{+\infty\}$  be a strictly convex potential with $\uU(\vec0)=0$ and recall the definition \eqref{gh} of  
the Hamiltonian $H_{\Lambda}\colon
S \to \R$.

Fix $\alpha>0$. The Hamiltonian determines 
a family of probability measures called specifications, indexed by the set of
finite $\Lambda\subset\Z^d$ and permutations $\xi$, defined by
\begin{equation}
\label{specif}
G_{\Lambda|\xi} (\sigma) := \frac1{Z_{\Lambda|\xi}}\exp(-\alpha
H_{\Lambda}(\sigma)) \,,\quad\s\in S_{\Lambda|\xi},
\end{equation}
where $Z_{\Lambda|\xi}$ is the normalizing constant
$Z_{\Lambda|\xi}:=\sum_{\sigma \in S_{\Lambda|\xi}}\exp(-\alpha
H_{\Lambda}(\sigma))$. Denote $G_\Lambda:=G_{\Lambda|\rI}$.

A measure $\mu$ on $S$ is said to be Gibbs at temperature $\alpha$ for the family of 
specifications $(G_{\Lambda|\xi})$ if the conditional distribution of $\mu$ on
$\Lambda$ given $\xi$ outside $\Lambda$ coincides with the 
specification $G_{\Lambda|\xi}$. That is, for finite $\Lambda\subseteq\Zd$ and
$\xi\in S$,
\[
\mu\big( \cdot\,|\,\s(x)=\xi(x), x\in \Lambda^c \big) = G_{\Lambda|\xi}
\quad \mbox{for $\mu$-almost all $\xi\in S$.} 
\]
We denote the set of Gibbs measures at temperature $\alpha$ by $\cG^\alpha$, 
and let ${\cG}=\cup_{\alpha>0} {\cG}^{\alpha}$.

Take $n\ge 2$. A \emph{finite cycle} $\gamma$ of length $|\gamma|=n$ associated to the set of distinct
sites $x_1,\dots,x_n$ is a permutation $\gamma\in S$ such that $\gamma(x)=x$ for
all $x\notin\{x_1,\dots,x_n\}$, $x_{i+1} = \gamma(x_i)$ for all $i \in \{1,\dots,n\}$, with the convention $x_{n+1}=x_1$.
An \emph{infinite cycle} $\gamma$ associated to a
doubly infinite sequence of distinct sites $\dots,x_{-1},x_0,x_1,\dots$ is a permutation
such that $\gamma(x) = x$ if $x\neq x_i$ for any $i$ and $x_{i+1}
= \gamma(x_i)$ for all $i$. The support of a cycle $\gamma$ associated to $x_1,\dots,x_n$ is
$\{\gamma\}= \{x_1,\dots,x_n\}$. 
Denote the set of finite cycles by
\begin{align}
  \Gamma:=\{\gamma\in S:\gamma \hbox{ is a cycle with  $\{\gamma\}$ finite}\}\quad \hbox{ and }\quad  
\Gamma_\Lambda:=\{\gamma\in \Gamma:\{\gamma\}\subset \Lambda\},
\end{align}
the set of cycles with support contained in $\Lambda$.
We say that two permutations are
\emph{disjoint} if their supports are so.  

Denote $\sigma\sigma'$ the composition of the  permutations $\sigma,\sigma'$:
\[
(\sigma\sigma')(x) := \sigma(\sigma'(x)).
\]
Any
permutation $\sigma\neq \rI$ can be written as a finite or countable composition of disjoint cycles:
\begin{equation*}
  \label{cycle-reresentation}
  \sigma = \dots\gamma_2\gamma_1,\qquad
  \{\gamma_i\}\cap\{\gamma_j\}=\emptyset,\,\,\,\hbox{for all } i\neq j\,,
\end{equation*}
note that the order of the cycles in this composition does not matter. The
identity has no cycle decomposition.  A permutation $\sigma$ is called
\emph{finite-cycle}  if all cycles in its decomposition are finite.
In this case we identify  $\sigma\ne\rI$
with the ``gas of cycles'' $\{\gamma_1,\dots,\gamma_k\}$, $k=k(\sigma) \in \N\cup \{+\infty\}$, while
the identity $\rI$ is identified with the empty set. We denote $\gamma\in\sigma$
when $\gamma$ is one of the cycles in the decomposition of $\sigma$. 

For a finite cycle $\gamma\in\Gamma$, define the \emph{weight} of $\gamma$ by
\begin{align}
\label{weight}
  w(\gamma) :=\; \exp\Bigl\{-\alpha\sum_{x\in\{\gamma\}} \uU\big(\gamma(x)-x\big)\Bigr\}. 
\end{align}
Since $\gamma$ is a cycle and $\uU$ is strictly convex, the sum in \eqref{weight} is strictly positive, which in turn implies  $w(\gamma)\in (0,1)$ for all $\alpha>0$. Define
\begin{align}
\label{beta}
  \beta(\uU,\alpha):= \sum_{\gamma\in\Gamma,\{\gamma\}\ni\vec 0} |\gamma|\,w(\gamma).
\end{align}
If $\beta(\uU,\alpha)$ is finite for some $\alpha$, then  $\beta(\uU,\alpha)$ is decreasing in $\alpha$ and  $\beta(\uU,\alpha)<1$ for all $\alpha>\alpha^*$ defined by
\begin{align}
\label{alpha*}
  \alpha^*(\uU) := \inf\{\alpha: \beta(\uU,\alpha)<1\}.
\end{align}
If $\beta(V, \alpha)=\infty$ for all $\alpha$ we set $\alpha^*=\infty$.

In our first theorem we give sufficient conditions on $\alpha$ for the existence of a Gibbs measure as limit of specifications with identity boundary conditions. The proof follows the lines proposed in \cite{MR1849182} to construct the infinite volume limit of the contour representation for the Ising model at low temperature. We include the proof for the convenience of the reader. 
\begin{theorem}
  \label{t1}
Identity boundary conditions.

\noindent Fix a strictly convex potential $\uU:\Zd\to\R^+\cup \{+\infty\}$ satisfying $\uU(\vec0)=0$. 
Assume $\alpha^*(\uU)<\infty$. Then, for each $\alpha>\alpha^*(\uU)$  there exists a random process $(\zeta_{\Lambda},\,\Lambda\subseteq \Z^d)$ on $(S_{\Lambda},\,\Lambda\subseteq \Z^d)$ such that

\noindent (i) for finite $\Lambda$, $\zeta_{\Lambda}$ is distributed according to $G_{\Lambda}$, the specification with identity boundary conditions, 

\noindent (ii) $\lim_{\Lambda\nearrow\Z^d}\zeta_{\Lambda}(x) = \zeta_{\Z^d}(x)$ almost
surely, for each $x \in \Z^d$. Call $\mu$ the distribution of $\zeta_{\Z^d}$. Then, $\lim_{\Lambda\nearrow\Z^d}G_{\Lambda}=\mu$ weakly. 

\noindent (iii) $\mu$ is an ergodic Gibbs measure at temperature $\alpha$ with mean jump~$\vec 0$.

\noindent (iv) $\mu$ is the unique Gibbs measure for the specifications $G_{\Lambda}$, supported on the set of finite-cycle permutations of $\Z^d$.
\end{theorem}

We next consider more general boundary conditions. 

We will say that the permutation $\sigma'$ is a \emph{local perturbation} of $\sigma$ if the set
$\{x\in\Z\,:\, \sigma'(x)\neq\sigma(x)\}$ is finite; in this case, the energy
difference between $\sigma'$ and $\sigma$ is defined by
  \begin{equation*}\label{diff}
    H(\s')-H(\s)  :=  \sum_{x:\sigma(x)\neq \sigma'(x)} \big(\uU(\sigma'(x)-x)-\uU(\sigma(x)-x)\big).
  \end{equation*}
A \emph{ground state} is a permutation $\xi\in S$ such that for any
local perturbation $\xi'$ of $\xi$, $H(\xi')-H(\xi) \;\geq\;0$. For $v\in\Zd$, the shift permutation $\tau_v\in S$ defined in \eqref{xv1} is a ground state: given a finite cycle $\gamma$, the permutation $\tau_v\gamma$ is a local perturbation of $\tau_v$ with energy difference
\begin{align}
  \label{tauvp}
H(\tau_v\gamma) - H(\tau_v) = \sum_{x\in\{\gamma\}} \big[\uU(\gamma(x)+v-x)-\uU(v)\big] > 0,
\end{align}
by the strict convexity of $\uU$. 

\begin{figure}[t]
\begin{center}
\psset{unit=1.2cm}
\begin{pspicture}(8,9)(0.5,1)  
\psset{linewidth=1.2pt,arrowsize=7pt,arrowlength=0.7, arrowinset=0.2}
\psset{linecolor=mygray}
\rput{0}(1,2){
\multido{\nx=0.0+0.5}{6}{
	\multido{\ny=-0.0+0.5}{6}{
		\pscircle[fillstyle=solid,fillcolor=black,linecolor=black](\nx,\ny){0.043}}}
\rput{0}(1.25,-0.5){$I$}
}
\rput{0}(5,2)
{
\multido{\nx=0.0+0.5}{6}{
	\multido{\ny=-0.0+0.5}{6}{
		\pscircle[fillstyle=solid,fillcolor=black,linecolor=black](\nx,\ny){0.043}}}
\rput{0}(1.25,-0.5){$\gamma$}
\psline{->}(1,0.5)(0.5,1.5)
\psline{->}(0.5,1.5)(1.5,1)
\psline{->}(1.5,1)(1,0.5)
}
\rput{0}(1,6){
\multido{\nx=0.0+0.5}{6}{
	\multido{\ny=-0.0+0.5}{6}{
		\pscircle[fillstyle=solid,fillcolor=black,linecolor=black](\nx,\ny){0.043}}}
\rput{0}(1.25,-0.5){$\tau_v$, $v=(1,1)$}
\multido{\nx=0.0+0.5}{5}{
	\multido{\ny=0.0+0.5}{5}{\rput{0}(\nx,\ny){
		\psline{->}(0,0)(0.5,0.5)}}}
}
\rput{0}(5,6)
{
\multido{\nx=0.0+0.5}{6}{
	\multido{\ny=-0.0+0.5}{6}{
		\pscircle[fillstyle=solid,fillcolor=black,linecolor=black](\nx,\ny){0.05}}}
\rput{0}(1.25,-0.5){$\tau_v\gamma$}
\multido{\nx=0.0+0.5}{5}{\rput{0}(\nx,0.0){\psline{->}(0,0)(0.5,0.5)}}
\multido{\nx=0.0+0.5}{2}{\rput{0}(\nx,0.5){\psline{->}(0,0)(0.5,0.5)}}
\multido{\nx=1.5+0.5}{2}{\rput{0}(\nx,0.5){\psline{->}(0,0)(0.5,0.5)}}
\multido{\nx=0.0+0.5}{3}{\rput{0}(\nx,1.0){\psline{->}(0,0)(0.5,0.5)}}
\rput{0}(2.0,1){\psline{->}(0,0)(0.5,0.5)}
\rput{0}(0.0,1.5){\psline{->}(0,0)(0.5,0.5)}
\multido{\nx=1.0+0.5}{3}{\rput{0}(\nx,1.5){\psline{->}(0,0)(0.5,0.5)}}
\multido{\nx=0.0+0.5}{5}{\rput{0}(\nx,2){\psline{->}(0,0)(0.5,0.5)}}
\pscurve{->}(1,0.5)(0.7,1.25)(1,2)
\pscurve{->}(0.5,1.5)(1.25,1.75)(2,1.5)
}
\end{pspicture}
\end{center}
\caption{Local perturbation $\tau_v \gamma$ of $\tau_v$, $v=(1,1)$, introduced by the cycle $\gamma$.}
\label{f30}
\end{figure}
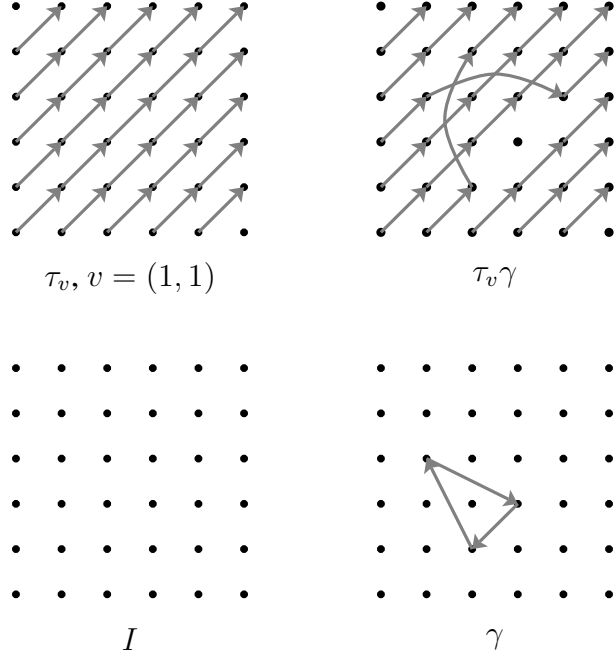

The next lemma says that a local perturbation of $\tau_v$ is a composition of a finite number of finite cycles with $\tau_v$. We leave the proof to the reader. See Figure~\ref{f30}.

\begin{lemma}
\label{tau'}
  If $\tau'_v$ is a local perturbation of $\tau_v$, then there exist disjoint finite cycles $\gamma_1,\dots,\gamma_n$ in $\Gamma$ such that $\tau'_v=\tau_v\gamma_1\dots\gamma_n$. If $\uU$ is strictly convex, then $H(\tau'_v) - H(\tau_v)>0$.  
\end{lemma}

In the following theorem we establish conditions on $\alpha$ that allow to extend the results of Theorem \ref{t1} to $\tau_v$-boundary conditions. 
For a finite cycle $\gamma\in\Gamma$, denote the $v$-\emph{weight} of $\gamma$ by
\begin{align}
\label{weightv}
  w_v(\gamma) :=\; \exp\Bigl\{-\alpha\sum_{x\in\{\gamma\}} \big(\uU(\gamma(x)+v-x\big)-\uU(v)\big)\Bigr\}. 
\end{align}
Given a measure $\mu$ and a permutation $\xi$ recall the definition of the shifted measure $\mu\xi$ from \eqref{comp}.
In order to obtain the result, we first consider  the composition of  a configuration with $\tau_v$ boundary conditions with the permutation $\tau_{-v}$ to produce a finite cycle permutation with cycles weighted by $w_v$. We then apply Theorem \ref{t1} to this random permutation, take limits in $\Lambda\nearrow \Z^d$, and as a last step compose the resulting measure with the permutation $\tau_v$ to recover the initial boundary conditions.

Let 
\begin{align}
\label{betav}
  \alpha^*_v(\uU) := \inf\{\alpha: \beta_v(\uU,\alpha)<1\},
\quad\hbox{ where }\quad
\beta_v(\uU,\alpha):= \sum_{\gamma\in\Gamma,\{\gamma\}\ni\vec 0} |\gamma|\,w_v(\gamma).
\end{align}

\begin{theorem}
  \label{t2}
  $\tau_v$ boundary conditions.
  
\noindent   Fix  a vector
  $v\in\Zd$ and a strictly convex potential $\uU:\Zd\to\R^+\cup \{+\infty\}$ such that $\uU(\vec0)=0$.
If $\alpha^*_v(\uU)$ is finite, then for any $\alpha>\alpha^*_v(\uU)$ there exists a random process $(\zeta_{\Lambda,v},\,\Lambda\subseteq \Z^d)$ on $(S_{\Lambda},\,\Lambda\subseteq \Z^d)$ such that

\noindent (i) For finite $\Lambda$, $\tau_v\zeta_{\Lambda,v}$ is distributed according to $G_{\Lambda|\tau_v}$, the specification with $\tau_v$ boundary conditions.

\noindent (ii) $\lim_{\Lambda\nearrow\Z^d}\tau_v\zeta_{\Lambda,v}(x) = \tau_v\zeta_{\Z^d,v}(x)$ almost surely, for all $x\in \Z^d$. Calling $\mu_v$ the law of $\tau_v\zeta_{\Z^d,v}$, we get $\lim_{\Lambda\nearrow\Z^d}G_{\Lambda|\tau_v}=\mu_v$ weakly. 

\noindent (iii) $\mu_v$ is an ergodic Gibbs measure at temperature $\alpha$ with mean jump~$v$.

\noindent (iv) $\mu_{v}\tau_{-v}$ is the only measure with cycle weights $w_v$
supported on the set of finite-cycle permutations of $\Z^d$. 

\end{theorem}


We finally consider separately the Gaussian potential. Although it is in principle covered by the previous results, it is worth pointing out that in this case the associated $v$-weights do not actually depend on $v$, $w_v(\gamma)=w(\gamma)$ for all $\gamma \in \Gamma$, with the consequence that the shift boundary condition measures are just the composition of the identity boundary conditions Gibbs measure $\mu$ with $\tau_v$, $\mu_v=\mu\tau_v$, and the value of $\alpha^*$ is the same for all $v\in \Z^d$, $\alpha_v^*=\alpha^*_{\vec{0}}$. We also compute an explicit bound on $\alpha^*$.


\begin{theorem}\label{t3}
  The Gaussian case.
  
\noindent Let $\uU(x)=\|x\|^2$, then $\alpha^*(V)\le \big(1.44504^{1/d}-1\big)^{-2}$. 
  
\noindent Fix $\alpha>\alpha^*(\uU)$, let $(\zeta_{\Lambda},\,\Lambda\subseteq \Z^d)$ be the process constructed in Theorem \ref{t1}, and let $\mu$ be the distribution of $\zeta_{\Z^d}$. 
 Then, for each $v\in\Z^d$,

\noindent (i) for finite $\Lambda$,  $G_\Lambda\tau_v= G_{\Lambda|\tau_v}$, the specification with $\tau_v$ boundary conditions. In particular, $\tau_v\zeta_\Lambda$ has distribution   $G_{\Lambda|\tau_v}$.

\noindent (ii) $\lim_{\Lambda\nearrow\Z^d}\tau_v\zeta_\Lambda(x) = \tau_v\zeta_{\Z^d}(x)$ almost surely, for all $x\in \Z^d$. As a consequence   $\lim_{\Lambda\nearrow\Z^d}G_{\Lambda|\tau_v}=\mu\tau_v$ weakly.

\noindent (iii) $\mu\tau_v$ is an ergodic Gibbs measure at temperature $\alpha$ and mean jump~$v$. 
\end{theorem}


\subsection{Sketch of the proofs}
\label{s2.1} 

 {\it Identity boundary conditions.} Consider a finite
$\Lambda\subset \Z^d$ and recall $S_\Lambda=S_{\Lambda|I}$ is the set
of permutations that equal the identity outside of $\Lambda$. 

A finite-cycle permutation $\sigma\in S$ can be identified with the
configuration $\eta\in\{0,1\}^\Gamma$ defined by
$\eta(\gamma)=\one\{\gamma\in\sigma\}$. Thus $S_\Lambda$ can be described as a subset of $\{0,1\}^{\Gamma_\Lambda}$:
\begin{align}
\label{iii}
  S_\Lambda= \Big\{\eta\in\{0,1\}^{\Gamma_\Lambda}: 
\eta(\gamma)\eta(\gamma')=0\hbox{ if }\{\gamma\}\cap\{\gamma'\}\neq\emptyset,\;\hbox{ for all } \;\gamma,\gamma'\in \Gamma_\Lambda
\Big\}\,.
\end{align}
Recall the definition \eqref{weight} of weight of a cycle $\gamma$. 
The specification in $\Lambda$ with identity boundary conditions \eqref{specif} can now be written as
\begin{equation}
  \label{mulambdaeta}
G_{\Lambda}(\eta)= \frac{1}{Z_\Lambda}\prod_{\gamma\in\Gamma_\Lambda}
w(\gamma)^{\eta(\gamma)},\qquad \eta\in S_\Lambda. 
\end{equation}

We interpret the  measure $G_\Lambda$  as the distribution of the gas of
cycles with weights $w$ and interacting by exclusion. This is the setup proposed in
\cite{MR1849182} to study the contour representation of the low temperature
Ising model.

Let now  $S^o= \{0,1,\dots\}^\Gamma$. Note that in $S^o$ cycles may have intersecting support; indeed, the same cycle 
may have multiplicity larger than 1.
Given a configuration $\eta\in S^o$, $\eta(\gamma)$ counts the
number of times the cycle $\gamma$ is present in $\eta$. Let $\mu^o$ be the
product measure on $S^o$ with marginal Poisson$(w(\gamma))$ for each $\gamma \in \Gamma$. If $\eta^o$ has
law $\mu^o$, then the random variable $\eta^o(\gamma)$ is Poisson with mean
$w(\gamma)$, and the random variables $\eta^o(\gamma)$, $\gamma\in\Gamma$ are independent.
For finite $\Lambda$,
$G_\Lambda$ is just the law $\mu^o$ conditioned to $S_\Lambda$:
\begin{align}
  \label{tl1}
G_\Lambda = \mu^o(\cdot\,|S_\Lambda)\,.
\end{align}
We claim that for large enough $\alpha$ we can construct a
Poisson measure on $S^o$ conditioned to the event that each cycle is present at most once, and present cycle supports are disjoint. That is, the measure is supported on the set of configurations
associated to finite-cycle permutations of~$\Z^d$. Since this set has
zero $\mu^o$-probability, an argument is required to give a proper
sense to this notion. 
For $\alpha$ large we construct $\mu$ as the invariant measure
for a continuous-time birth and death process of cycles interacting by exclusion, and show that it
concentrates on finite-cycle permutations. We also prove that $\mu$ is the limit as $\Lambda\to\infty$ of $G_\Lambda$ given by \eqref{tl1}.

Given a cycle $\gamma\in\Gamma$, consider the rates of a continuous-time birth and
death process on $\{0,1,\dots\}$ defined by
\begin{align}
  \label{qg}
\hbox{birth rates: }q_\gamma(k,k+1):=w(\gamma),\;\; \hbox{death rates: }q_\gamma(k+1,k):=k+1,\;\; k\ge 0.
\end{align}

We construct birth and death processes with the above rates as a function of a Poisson process. Let $\cN$ be a Poisson process on $\Gamma\times\R\times\R^+$ with rate
measure $w(\gamma)\times dt\times e^{-s}ds$. If the
point $(\gamma,t',s')\in\cN$, we say that a cycle $\gamma$ is born at time $t'$ and lives
until $t'+s'$.  Define $\eta^o$
as the number of cycles $\gamma$ \emph{alive} at time $t$.
By construction $(\eta^o_t(\gamma),\,t\in\R)$ is a
time-stationary continuous-time birth and death process with rates $q_\gamma$ given in \eqref{qg}; that is, at any time a new copy of a cycle $\gamma$ is born at rate
$w(\gamma)$, whereas existing copies die independently at rate~$1$. The marginal distribution of $\eta^o_t(\gamma)$ is Poisson with mean $w(\gamma)$, for each $t\in\R$.
Letting $\eta^o_t:=(\eta^o_t(\gamma):\gamma\in\Gamma)$, the process $(\eta^o_t,\,t\in\R)$ is a family of stationary independent birth and death processes with 
marginal distribution $\mu^o$ at any time $t$. 

Our goal is to
perform such a graphical construction for a birth and death process
with the same rates, subject to an exclusion rule as follows. Now
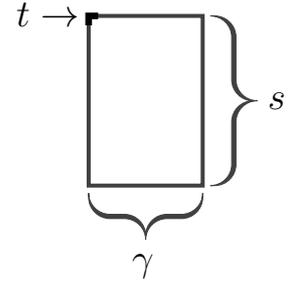
\begin{wrapfigure}{r}{0.25\textwidth}
\vspace{-20pt}
\begin{center}
\psset{unit=0.5cm}
\begin{pspicture}(12,8)(5,0.5)
\psset{linewidth=1.5pt,arrowsize=12pt}
\large
\psset{linecolor=darkgray,linewidth=1.5pt}
\pspolygon(7,2.5)(10,2.5)(10,7)(7,7)
\rput[r]{0}(6.75,7.05){$t\rightarrow$}
\psline[linewidth=2.5pt,linecolor=black](7,6.75)(7,7)(7.25,7)
\psbrace[linewidth=0.05pt,fillcolor=darkgray,nodesepA=4pt,ref=lC](10.2,2.5)(10.2,7){$s$}
\psbrace[linewidth=0.05pt,fillcolor=darkgray,nodesepB=10pt,nodesepA=-1pt,rot=90,ref=c](7,2.3)(10,2.3){$\gamma$}
\end{pspicture}
\end{center}
\caption{The representation of the point $\varphi=(\gamma,t,s)$. Time is going down.}
\label{f40}
\vspace{-10pt}
\end{wrapfigure}
the point $(\gamma,t,s)\in\cN$ represents a birth attempt of a
cycle $\gamma$ at time $t$ (see Figure~\ref{f40}), but the cycle will be effectively born
only if its support $\{\gamma\}$ does not intersect the support of any
of the cycles already present at that time $t$. When the process is
restricted to a finite set $\Lambda$, the points in
$\{(\gamma,t,s)\in\cN$, $\gamma\in\Gamma_\Lambda\}$ can be ordered by
their birth time $t$. Since the free process is empty infinitely
often: $\eta^o_t(\gamma)= 0$ for all $\gamma\in\Gamma_\Lambda$ for
infinitely many positive and negative times, it is possible to
iteratively decide for each $(\gamma,t,s)$ if it actually produces a
birth of $\gamma$ in the model with exclusion, or not. We so construct
a stationary birth and death process $(\eta^\Lambda_t,\,t\in\R)$ on
$\Gamma^\Lambda$ with rates $(q_\gamma,\gamma\in\Gamma^\Lambda)$
subjected to the exclusion condition on cycles in $\Lambda$. The
marginal distribution of $\eta^\Lambda_t$ is $G_\Lambda$.

In infinite volume the above argument does not work because the
configuration is never empty.  Instead, for each point
$(\gamma,t,s)\in\cN$ one can look for the points of $\cN$ born prior
to $t$ that could interfere with the birth of the cycle $\gamma$ at
time $t$. This set is called the \emph{clan of ancestors} of $(\gamma,
t,s)\in\Gamma\times\R\times\R^+$. If the clan of ancestors of any point is finite with probability
one, then it is possible to construct the stationary loss network of
finite cycles in $\Z^d$. We call $(\eta_t,\, t\in\R)$ the resulting
Markov process, obtained as a deterministic function of $\cN$. Let us
suggestively denote by $\mu$, the notation previously used to name the
Gibbs measure, the distribution of the permutation with cycles
$\eta_t$ for a given time $t$.  Since the construction is
time-stationary, the measure $\mu$ does not depend on $t$: it is an
invariant measure for the process. In fact one can check that $\mu$
is reversible for the process. We show that $\mu$ is the thermodynamic
limit of $G_\Lambda$ and the unique invariant measure for the
process~$(\eta_t)$.

In order to prove that $\mu$ is the thermodynamic limit of $G_\Lambda$, we construct a stationary family of processes
$(\eta^\Lambda_t,\,t\in\R)$ for any $\Lambda
\subset \Z^d$ as a function of a unique realization $\cN$ of the
Poisson process; a coupling. For finite $\Lambda$, the marginal
distribution of $\eta^\Lambda_t$ is $G_\Lambda$.  We use the
finiteness of the clan of ancestors to show that for each finite-cycle
$\gamma$, $\eta^\Lambda_t(\gamma)$ converges to $\eta_t(\gamma)$ as
$\Lambda\nearrow\Z^d$, for almost all realizations of the point
process $\cN$. In particular, this proves that $G_\Lambda$ converges
weakly to $\mu$ and yields several properties of the limit.

To show that the clan of ancestors of a point $(\gamma,t,s)$ is finite
we dominate it by a multitype branching process and then show that the condition $\beta(\uU,\alpha)<1$ is a sufficient condition for the branching process to die out. We give more details of these processes
in Section \ref{s3}.  

For any fixed $t\in\R$, the process $(\eta^\Lambda_t: \Lambda\subseteq \Z^d)$ satisfies the properties attributed to the process $(\zeta_\Lambda,\,\Lambda\subseteq \Z^d)$ in Theorem \ref{t1}.
 \\
 
\emph{$v$-jump boundary conditions. } The specifications associated to the potential 
\begin{align}
\label{vv}
\uU_v(y):= \uU(y+v)-\uU(y)
\end{align}
 with identity boundary conditions are given by
$G_{\Lambda|\tau_v}\tau_{-v}$ on  $S_\Lambda$, for any finite $\Lambda\subseteq \Z^d$. We prove that for $\alpha^*_v(\uU)<\infty$, $\uU_v$ satisfies the conditions of Theorem \ref{t1} to obtain a process $(\zeta_{\Lambda,v},\,\Lambda\subset\Z^d)$ on $(S_{\Lambda},\,\Lambda\subseteq\Z^d)$ with the properties stated in that theorem.  We then use the fact that $\tau_v\zeta_{\Lambda,v}$ has law $G_ {\Lambda|\tau_v}$ to obtain Theorem \ref{t2}. \\

\emph{$v$-jump boundary conditions for the Gaussian potential. }  When the potential  is Gaussian, $\uU=\|\cdot\|^2$, we have  $G_{\Lambda|\tau_v}=G_{\Lambda}\tau_v$, a fact proven in Section \ref{s4}. This is the key to  the proof of Theorem~\ref{t3}.



\section{Loss networks of finite cycles}\label{s3}

We here construct the invariant measure
of a loss network of cycles and show that it is the Gibbs measure related to the specifications
$G_\Lambda$ with identity boundary conditions. The section consists of a review of
\cite{MR1849182} described in terms of cycles instead of contours.

\paragraph{Loss network}
Take a potential $V$ and a set $\Lambda \subset \Z^d$. Recall the definition \eqref{iii} of $S_\Lambda\subset\{0,1\}^{\Gamma_\Lambda}$. We introduce a continuous-time Markov process in  $S_\Lambda$
called \emph{loss network} of finite cycles. We say that two cycles are
\emph{compatible} if their supports are disjoint. Given a configuration
$\eta\in S_\Lambda$ of the process, we add a new cycle $\gamma$ at rate
$w(\gamma)$, if it is compatible with $\eta$, that is, if $\gamma$ is compatible
with all cycles $\gamma'$ with $\eta(\gamma')=1$. If $\gamma$ and $\eta$ are not
compatible, then the cycle is not added and the attempt is lost, hence the name
loss network.  Finally, any cycle in $\eta$ is deleted at rate one.  Loss networks
were introduced as stochastic models of a telecommunication network in which
calls are routed between nodes around a network.  In our case the nodes are the $d$-dimensional integers and a call uses the nodes in the support of a
non-identity cycle. Each node has capacity to support at most one
call and hence arriving calls that would occupy an already busy node are lost.  An
account of the properties of loss networks can be found in Kelly
\cite{MR1111523}.

Denote $\gamma\sim\eta$ if 
 $\gamma$ is compatible with $\eta$; in particular $\gamma\sim\eta$
implies $\eta(\gamma)=0$.
The loss network process on $S_\Lambda\subset\{0,1\}^{\Gamma_\Lambda}$ has formal generator 
\begin{equation}
  \label{generator-loss}
  {\cL}^\Lambda f(\eta) = \sum_{\gamma\in \Gamma_\Lambda} w(\gamma)
  \one_{\{\gamma\sim\eta\}}\,[f(\eta+\delta_\gamma)-f(\eta)] +
  \sum_{\gamma\in\Gamma_\Lambda} \,[f(\eta-\delta_\gamma)-f(\eta)],
\end{equation}
where $f$ is a test function, and $\delta_\gamma(\gamma')= 1$ if and only if
$\gamma'=\gamma$. 
When $\Lambda$ is finite, the
loss network is a well defined, irreducible Markov process on a finite state space, with a unique
invariant measure.  

The next
lemma shows that $G_\Lambda$ defined in  \eqref{mulambdaeta} is reversible for the loss network $\eta^\Lambda_t$; the proof is left to the reader. 
\begin{lemma}
  \label{finite-reversible} Let $\Lambda$ be finite. The measure $G_{\Lambda}$  is reversible for the dynamics
\eqref{generator-loss}. In particular, this is the unique invariant measure, and the weak
limit of the distribution of the process starting from any initial permutation as $t\to\infty$.
\end{lemma}

In the following we show that when 
$\alpha>\alpha^*(\uU)$ given in \eqref{alpha*} there exists a stationary process with generator
\eqref{generator-loss} for any $\Lambda\subseteq\Z^d$. The proof relies on a coupling argument applying the Harris
graphical construction of the process: to each configuration of an appropriate Poisson process $\cN$
we associate a realization of the loss network, $\cN\mapsto (\eta^\Lambda_t)$, for any $\Lambda\subseteq\Z^d$. 
We now introduce the basic elements of the argument.

\emph{The Poisson process. } Let $\cN$ 
be a Poisson process on $\Gamma\times\R\times\R^+$ with intensity
measure 
\[
d(\gamma,t,s) = w(\gamma)\, dt\, e^{-s} ds\,.
\] 
This process can be thought of as a product of independent Poisson processes on
$\R\times \R^+$, indexed by $\gamma\in\Gamma$. 

\emph{The free process. } 
Given the Poisson process $\cN$, define the {\em free process} $(\eta^o_t, t\in \R)$ on $\{0,1,\dots\}^\Gamma$ by
\begin{equation}
  \label{stationary-free-process}
  \eta^o_t(\gamma):= \sum_{(\gamma,t',s')\in\cN} \one\{ t'\le t < t'+ s'\}\,.
\end{equation}
If a point $(\gamma,t,s)\in\cN$, we say that a cycle $\gamma$ is born
at time $t$ and lives $s$ time units. We represent it as a cylinder with base
$\gamma$, height $s$ with its higher point located at $t$. See Figure~\ref{f40} where
the basis is represented by a segment.

The construction implies that cycles of type $\gamma$ are born independently at rate
$w(\gamma)$, and each of them lives for an exponential time of parameter 1; there
may be more than one cycle of type~$\gamma$ present at any given time. The process
$\eta^o_t$ is thus obtained as the product of independent birth and death processes $\big(\eta^o_t(\gamma):{\gamma \in \Gamma}\big)$, with birth rates $w(\gamma)$ and death rate~1.  The generator of
$\eta^o_t$ is given by
 \begin{equation*}\label{free}
\cL^o f(\eta) = \sum_{\gamma\in \Gamma} w(\gamma)
\,[f(\eta+\delta_\gamma)-f(\eta)] +
\sum_{\gamma\in\Gamma}\eta(\gamma) \,[f(\eta-\delta_\gamma)-f(\eta)]\,,
\end{equation*}
where $f:\{0,1,\dots\}^\Gamma \to\R$ is any local test function in the domain of
${\cL}^{o}$.
It is easy to see that the product measure $\mu^o$ on $\{0,1,\dots\}^\Gamma$ with Poisson marginals
\[
 \mu^o(\eta:\eta(\gamma)=k) = \frac{e^{-w(\gamma)}(w(\gamma))^k}{k!}
\]  
is reversible for the free process. Indeed, this is the law of the
configuration $\eta^o_t$ defined in \eqref{stationary-free-process}, for any
fixed $t\in\R$. 

\emph{The clan of ancestors. }
We will construct a stationary version of the loss network in infinite volume starting from
the stationary free process, by simply erasing those cycles that violate the exclusion
condition at birth. In order to make sense of this construction we need to consider the clan of ancestors of
each point $(\gamma, t, s)\in\Gamma\times\R\times\R^+$, as follows. 

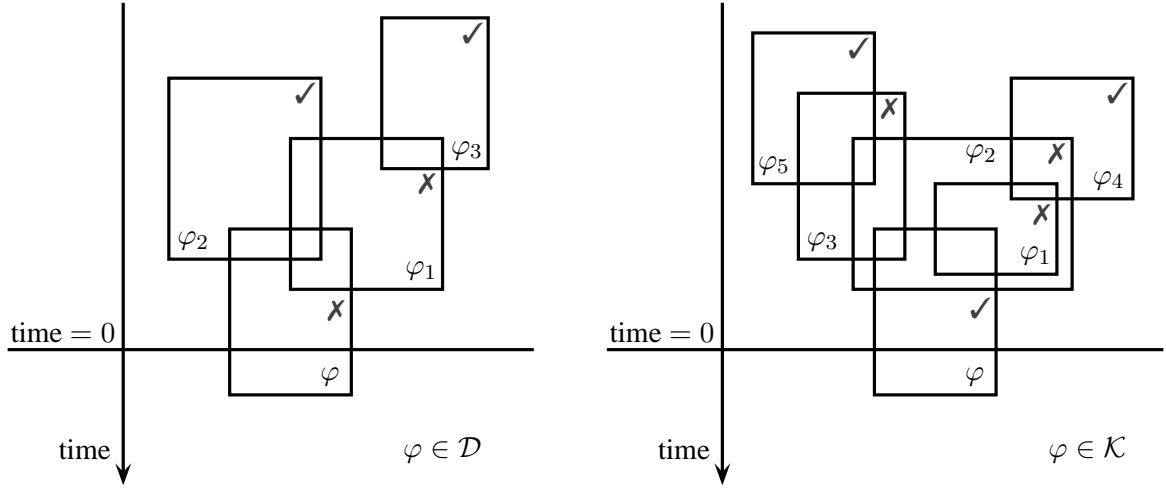
\begin{figure}
\begin{center}
\psset{unit=0.4cm}
\begin{pspicture}(40,16)(0,-1)  
\small
\psset{linewidth=1.2pt,arrowsize=6pt}
\rput{0}(2.3,0){
\psline(-1.3,4)(16,4)
\psline{->}(2.5,15.5)(2.5,-0.5)
\rput[rb]{0}(2.25,4.25){time $=0$}
\rput[r]{0}(2.1,0.7){time}
\pspolygon(6,2.5)(10,2.5)(10,8)(6,8)
\rput[B]{0}(9.3,3){$\varphi$}
\rput{0}(9.5,5.3){\color{darkgray}\ding{55}}
\pspolygon(8,6)(13,6)(13,11)(8,11)
\rput[B]{0}(12.3,6.5){$\varphi_1$}
\rput{0}(12.5,9.5){\color{darkgray}\ding{55}}
\pspolygon(4,7)(9,7)(9,13)(4,13)
\rput[B]{0}(4.8,7.5){$\varphi_2$}
\rput{0}(8.5,12.5){\color{darkgray}\ding{51}}
\pspolygon(11,10)(14.5,10)(14.5,15)(11,15)
\rput[B]{0}(13.8,10.5){$\varphi_3$}
\rput{0}(14,14.5){\color{darkgray}\ding{51}}
\rput{0}(13,0.7){$\varphi\in\bD$}
}

\rput{0}(22,0){
\psline(-1.3,4)(17,4)
\psline{->}(2.5,15.5)(2.5,-0.5)
\rput[rb]{0}(2.25,4.25){time $=0$}
\rput[r]{0}(2.1,0.7){time}
}
\rput{0}(23.5,0){
\pspolygon(6,2.5)(10,2.5)(10,8)(6,8)
\rput[B]{0}(9.3,3){$\varphi$}
\rput{0}(9.5,5.4){\color{darkgray}\ding{51}}
\pspolygon(3.5,7)(7,7)(7,12.5)(3.5,12.5)
\rput{0}(4.3,7.6){$\varphi_3$}
\rput{0}(6.5,12){\color{darkgray}\ding{55}}
\rput{0}(-2,-3.5){
\pspolygon(10,10)(14,10)(14,13)(10,13)
\rput[B]{0}(13.3,10.5){$\varphi_1$}
\rput{0}(13.5,12){\color{darkgray}\ding{55}}}
\rput{0}(-0.5,-1){
\pspolygon(11,10)(15,10)(15,14)(11,14)
\rput[B]{0}(14.2,10.5){$\varphi_4$}
\rput{0}(14.5,13.5){\color{darkgray}\ding{51}}}
\rput{0}(-9,-0.5){
\pspolygon(11,10)(15,10)(15,15)(11,15)
\rput[B]{0}(11.7,10.5){$\varphi_5$}
\rput{0}(14.5,14.5){\color{darkgray}\ding{51}}}
\rput{0}(13,0.7){$\varphi\in\bK$}
}
\pspolygon(28.8,6)(36,6)(36,11)(28.8,11)
\rput{0}(33,10.5){$\varphi_2$}
\rput{0}(35.5,10.5){\color{darkgray}\ding{55}}
\end{pspicture}
\end{center}
\caption{The clan of ancestors of the point $\varphi=(\gamma,t,s)$ in two scenarios. On the left $\varphi$ is deleted, while on the right it is kept. }
\label{f50}
\end{figure}

The first generation of ancestors of  $\varphi=(\gamma,t,s)$ is the subset of $\cN$
defined by
\[
 \bA_1^\varphi:=\{(\gamma',t',s') \in \cN \colon \gamma'\not\sim\gamma, \,  t'<t<t'+s'\}.
\]
where, as before,  two cycles $\gamma$ and $\gamma'$ are incompatible,
$\gamma\not\sim\gamma'$, if their supports have non empty intersection; in
particular, a cycle is incompatible with itself: $\gamma\not\sim\gamma$. Iteratively, the
$(n+1)$-th generation of ancestors of $\varphi$ is the union of the first
generation of ancestors of the points belonging to the $n$-th generation of
ancestors of $\varphi$, that is,
\[
 \bA_{n+1}^\varphi:=\bigcup_{\varphi' \in \bA_n^\varphi} \bA_1^{\varphi'}.
\]
The \emph{clan of ancestors} of $\varphi$ is the union of all generations of
ancestors: 
\begin{equation}
\label{clan}
\bA^\varphi:=\cup_{n\ge 1}\bA_n^\varphi.
\end{equation}
See Figure~\ref{f50} with two scenarios. On the left $\bA_1^\varphi
=\{\varphi_1,\varphi_2\}$ and $\bA_2^\varphi=\{\varphi_3\}$. On the right $\bA_1^\varphi
=\{\varphi_1,\varphi_2,\varphi_3\}$,  $\bA_2^\varphi=\{\varphi_2, \varphi_3, \varphi_4,\varphi_5\}$,  $\bA_3^\varphi=\{\varphi_3, \varphi_4,\varphi_5\}$ and
$\bA_4^\varphi=\{\varphi_5\}$.

\emph{Kept and deleted points. }  Assume $\bA^\varphi$ finite for all
$\varphi\in \Gamma \times \R\times \R^+ $, for almost all realizations of $\cN$. Fix $\bD_0=\emptyset$, and
for $n\ge1$ let 
\[
\bK_n := \{\varphi\in\cN\,:\,
\bA_1^\varphi\setminus \bD_{n-1}=\emptyset\}, \qquad 
\bD_n := \{\varphi\in\cN\,:\, \bA_1^\varphi\cap\bK_n\neq\emptyset\}.
\]
Let $\bK:=\cup_n\bK_n\subseteq \cN$ be the set of \emph{kept} points,
and $\bD:=\cup_n\bD_n\subseteq \cN$ be the set of \emph{deleted}
points. As a consequence of the finiteness of the clans of ancestors,
every point is either kept or deleted. Indeed, to determine whether a
point $\varphi$ is in $\bK$ or $\bD$, it suffices to inspect its clan
of ancestors $\bA^\varphi$.
In Figure \ref{f50} we have checked the kept points and crossed
the deleted ones.

\emph{Stationary loss network. } Assume $\bA^\varphi$ finite for all
$\varphi\in \Gamma \times \R\times \R^+ $, for almost all realizations of $\cN$.
 Define the \emph{stationary
loss network} $(\eta_t, t\in \R)$ by
\begin{equation}
\label{def.eta}
 \eta_{t}(\gamma):= \sum_{(t',s'): (\gamma,t',s')\in\cN} \one\{t'\le t<t'+s'\} \, \one\{(\gamma,t',s')\in\bK\}.
\end{equation}
This is the set of cycles associated to kept points alive at time $t$.
Note that $\eta_t(\gamma) \in
\{0,1\}$. 
The process $(\eta_t, t\in \R)$ is stationary by construction, let us call $\mu$ its stationary distribution, 
\begin{equation}
\label{mu}
\mu:= \hbox{ law of }\eta_t, \hbox {for any } t\in \R.
\end{equation}
The reader can prove the following result. 
\begin{proposition}
  \label{loss-stationary}
Assume $\bA^\varphi$ finite for all
$\varphi\in \Gamma \times \R\times \R^+ $, for almost all realizations of $\cN$.
 Then,
the process $(\eta_t,\,t\in\R)$ defined in \eqref{def.eta} is Markov with
generator \eqref{generator-loss} and invariant measure $\mu$ as in \eqref{mu}.
\end{proposition}

\emph{Thermodynamic limit. }The set of kept points is a deterministic function of $\cN$:
$\bK=\bK(\cN)$. Since the process $(\eta_t,\,t\in\R)$ is a function of the kept points, it
is also a function of $\cN$: $(\eta_t)=(\eta_t)(\cN)$. Given $\Lambda\subset\Z^d$ define the Poisson process 
associated to the cycles in $\Gamma_\Lambda$,
\begin{equation*}
  \label{nl}
  \cN^\Lambda:=\{(\gamma,t,s)\in\cN\,:\, \{\gamma\}\subset\Lambda\},
\end{equation*}
 the corresponding set of kept points 
$\bK^\Lambda:=\bK(\cN^\Lambda)$, and the loss network of cycles in $\Lambda$
\begin{equation}
  \label{el}
  (\eta^\Lambda_t):=(\eta_t)(\cN^\Lambda).
\end{equation}
Clearly $\cN^\Lambda$ is a function of $\cN$. When the clan of ancestors of any point is finite for almost all realizations of $\cN$, we have managed to define all 
processes $(\eta^\Lambda_t,\,t\in\R)$, $\Lambda\subseteq\Z^d$,  as a function of the same realization $\cN$ of the point process. In particular notice that $\eta^{\Z^d}_t=\eta_t$. 

When $\Lambda$ is finite, the finiteness of the clan of ancestors is guaranteed
and in this case $(\eta^\Lambda_t)$ is an irreducible Markov process in the
finite state space $S_\Lambda\subset\{0,1\}^{\Gamma_\Lambda}$ with generator $\cL^\Lambda$ given by
\eqref{generator-loss}.
By Lemma \ref{finite-reversible}, the distribution of $\eta^\Lambda_t$ is the
measure $G_\Lambda$, which is reversible for the process, for any $t\in\R$.

We now state and prove the thermodynamic limit.  

\begin{theorem}
  \label{tl}
 Existence of almost sure thermodynamic limit.
   
 \noindent Assume $\bA^\varphi$, the clan of ancestors of $\varphi$,  is finite for all $\varphi$ for almost all realizations of $\cN$. Then
  for any fixed $t\in\R$ and $\gamma\in\Gamma$, $\lim_{\Lambda\nearrow\Z^d} \eta_t^\Lambda(\gamma) =
  \eta_t(\gamma)$ almost surely. In particular, as $\Lambda\nearrow\Z^d$,
  $G_{\Lambda}$ converges weakly to $\mu$, the stationary  law of $\eta_t$ in \eqref{mu}.
\end{theorem}

\begin{proof}
 Take a realization $\cN$ such that $\bA^\varphi$ is finite
  for all $\varphi \in \Gamma\times \R\times \R^+$.  It suffices to show that for any $\gamma\in\Gamma$ and $t\in \R$,
  there exists a set $\Lambda_t(\cN,\gamma)$ such that if $\Lambda$ contains 
  $\Lambda_t(\cN,\gamma)$, then $\eta^\Lambda_t(\gamma) =
  \eta_t(\gamma)$. Take the point $(\gamma,t,s)$ ($s$ is irrelevant here) and define
  \begin{align}
    \Lambda_t(\cN,\gamma) := \bigcup_{(\gamma', t',s')\in \bA^{(\gamma,t,s)}}\{\gamma'\},
  \end{align}
the union of the supports of the cycles $\gamma'$ present in the clan of ancestors of $(\gamma,t,s)$.
  Now if $\Lambda$ contains  
 $\Lambda_t(\cN,\gamma)$, then the clan of ancestors restricted to $\Lambda$ is the same as the non-restricted clan: 
  $\bA^{(\gamma,t,s)}(\cN^\Lambda)= \bA^{(\gamma,t,s)}(\cN)$. This implies $\eta^\Lambda_t(\gamma)=\eta_t(\gamma)$ for all $\Lambda\supset\Lambda_t(\cN,\gamma)$. 
\end{proof}

\begin{theorem}
\label{uniqueness}
Uniqueness.

\noindent Assume $\bA^\varphi$ is finite for any point $\varphi=(\gamma,t,s)$, for almost all realizations of $\cN$. 
  Let $\nu$ be an invariant measure for the loss network dynamics defined by
  \eqref{generator-loss} supported on finite-cycle configurations.
  Then $\nu=\mu$, the law of the stationary process $\eta_t$ at any fixed time $t$. 
\end{theorem}

\begin{proof}
  Let $(\eta_t)_{t \in\R}$ denote the stationary loss
  network. Consider a family  $(s(\theta):\theta\in\Gamma)$ of iid random variables with exponential distribution of rate
  $1$. For any $\eta' \in \{0,1\}^\Gamma$ define
   \begin{align}
  \cI_{\eta'}(u):=\{(\theta,u,s(\theta)):\eta'(\theta)=1\}.
\end{align}
and the set of $\cN$-points born after $u$,
\begin{align}
  \cN_{[u,\infty)}=\{(\gamma,t',s)\in\cN: t'\ge u\}.
\end{align}
Let $(\eta'_{[u,t]}:{t\ge u})$ be the coupled loss network with initial configuration $\eta'_u=\eta'$ that
updates using the points in  $\cN_{[u,\infty)} $, and such that each initial cycle $\theta \in
  \eta'$ dies at time $u+s(\theta)$. Then 
\begin{align}
 ( \eta'_{[u,t]},\,t\ge u):= (\eta_t,\,t\ge u)(\cI_{\eta'}(u)\cup\cN_{[u,\infty)})\,.
\end{align}
Note that the distribution of $ \eta'_{[-t,0]}$ is the same as $\eta'_{[0,t]}$.

We can compare the latter process with the stationary process at time 0: for any $\gamma \in \Gamma$
\begin{align}
\label{compare}
  \big|\eta_{0}(\gamma) -\eta'_{[-t,0]}(\gamma)\big|\le 
\one\{\bA^{(\gamma,0,s)}\not\subset\cN_{[-t,\infty)}\}
+ \one\{\bA^{(\gamma,0,s)}\not\sim  \cI_{\eta'}(-t)\},
\end{align}
where the death time $s\in (\gamma, 0, s)$ is in fact irrelevant to the computation. Equation \eqref{compare} says that if each point in the clan of ancestors of $(\gamma,0,s)$ is born after time $-t$, and it is compatible with all points associated to the cycles $\theta$ in the initial configuration $\eta'$, then the cycle $\gamma$ belongs to both configurations or to none of them. 
Since the clan of ancestors is finite, for any cycle $\gamma$:
\begin{align}
  \lim_{t\to\infty} \big|\eta_{0}(\gamma) -\eta'_{[-t,0]}(\gamma)\big| = 0  \qquad a.s..
\end{align}
Sample a random $\eta'$ distributed according to the invariant measure $\nu$, then $\eta'_{[-t,0]}$ has law $\nu$ for all $t$ and $\eta'_{[-t,0]} (\gamma)\to \eta_{0}(\gamma)$ almost surely for all finite cycle $\gamma$. If $f:\{0,1\}^\Gamma \to\R$ is a bounded cylindrical function this implies $f(\eta'_{[-t,0]}) \to_{t\to\infty} f(\eta_{0})$ almost surely and
\begin{align}
  \nu f= \int \nu(d\eta') f(\eta'_{[-t,0]}) \;\mathop{\longrightarrow}_{t\to\infty}\;E f(\eta_{0})=\mu f\,,
\end{align}
i.e., $\nu=\mu$. 
\end{proof}

\paragraph{Conditions for the clan of ancestors to be finite}
The results of this section depend crucially on the hypothesis that the clan of ancestors of any point $\varphi \in \Gamma\times \R\times \R^+$ be finite for almost all realizations of 
$\cN$. 

Fix $\varphi=(\gamma,t,s)$. For each $\theta \in \Gamma$ and $n\ge 1$,  define 
\begin{equation*}
  \label{a1g}
  A_n(\gamma,\theta):=\big|\big\{(t',s')\,:\,(\theta,t',s')\in \bA_n^{(\gamma,t,s)}\big\}\big|,
\end{equation*}
the number of $\theta$-points $\in \cN$ that belong to the $n$-th generation of ancestors of 
$(\gamma,t,s)$, and denote the number of $\theta$-points in the clan of ancestors of $(\gamma,t,s)$ by 
\begin{align}
  A(\gamma,\theta):= \sum_{n\ge 0} A_n(\gamma,\theta). 
\end{align}
 The total size of the clan of ancestors is $|\bA^{(\gamma,t,s)}|=\sum_{\theta\in\Lambda}A(\gamma,\theta)$. We set conditions on $\uU$ and $\alpha$ that ensure this sum is finite.

\emph{Subcritical multitype branching process. }
We dominate the number of points in the clan of ancestors by
a branching process. Ancestors in the clan become descendants for the branching process, hence time runs backwards for the branching process.

Let $B_n$ be a discrete time
multitype branching process with type-space $\Gamma$ and offspring distribution
$A_1(\gamma,\theta)$. The number of children of type
$\theta$ in the $n$-th generation is defined by
$B_0(\gamma,\theta)=1\{\theta=\gamma\}$, and for $n\ge 0$,
\[
 B_{n+1}(\gamma,\theta) = \sum_{\gamma' \in \Gamma}\sum_{i=1}^{B_n(\gamma,\gamma')}A_{1,n+1,i}(\gamma',\theta)
\]
where $A_{1,n,i}(\gamma,\theta)$ are independent random variables with the same
distribution as $A_{1}(\gamma,\theta)$.  Let
$B(\gamma,\theta):=\sum_nB_n(\gamma,\theta)$ be the total number of descendants of type $\theta$ of a cycle $\gamma$. \begin{lemma}
\label{bb9}
$A(\gamma,\cdot)$ is stochastically dominated by $B(\gamma,\cdot)$.
\end{lemma}
\begin{proof}
  The branching process $B_{n+1}$ counts twice or more times those cycles $\theta$ in the
  $(n+1)$-th generation that intersect more than one $\gamma'$ on the $n$-th
  generation, while $A_{n+1}$ counts them only once. For details see \cite{MR1707339} and
  \cite{MR1849182}.
\end{proof}

We conclude that if the branching process is subcritical then the clan of ancestors is finite almost surely. 

\emph{Mean number of ancestors. } Let
\begin{equation*}
  \label{a2g}
  m(\gamma, \theta):=E\big[A_1(\gamma,\theta)\big].
\end{equation*}
By stationarity, the law of $A_1(\gamma,\theta)$ does not depend on $t$. Also, the property of 
being an ancestor of $(\gamma,t,s)$ is determined by the type $\gamma$ and its birth time $t$: $A_1(\gamma, \theta)$ does not depend on $s$.
The random variable 
$A_1(\gamma,\theta)$ has Poisson distribution with mean $m(\gamma, \theta)$. 

A point $(\theta,t',s')$
in the first generation of ancestors of $(\gamma,t,s)$ must satisfy $\theta\not\sim\gamma$,  $t'<t$ and  $s'\ge t-t'$. Hence,
\begin{equation}
  \label{bb3}
 m(\gamma, \theta)= w(\theta)\one\{\gamma \not\sim \theta\}
\int_{-\infty}^t\,dt' \int_{t-t'}^\infty\, ds' \, e^{-s'}= w(\theta)\one\{\gamma \not\sim \theta\}.
\end{equation}

Since $B_1(\gamma,\theta)$ has the same law as $A_1(\gamma,\theta)$, the mean matrix of the branching process is given by $m$ and the mean number of descendants of  type $\theta$ from an individual of type $\gamma$ 
  after $n$ branchings is given by
  \begin{align}
\label{mngt}
    E\big[B_{n}(\gamma,\theta)\big] = m^n(\gamma, \theta).
  \end{align}
\begin{lemma}
  \label{beta-alpha}
For fixed $\gamma\in\Gamma$ and $n\ge1$ the following inequality holds
\begin{equation}\label{malan}
 \sum_{n\ge 1}\sum_{\theta \in \Gamma} m^n(\gamma, \theta) \le \sum_{n\ge 1}|\gamma|\beta^n,
\end{equation}
where $\beta=\beta(\uU,\alpha)$ was defined in \eqref{beta}. In particular $\beta(\uU,\alpha)<1$ implies that the expected number of descendants of a finite cycle $\gamma$ is finite.
\end{lemma}

\begin{proof} Recall $m(\gamma,\theta)=\one\{\theta\not\sim\gamma\}\,w(\theta)$, and bound
\begin{eqnarray}
  \label{bb}
   \sum_{\theta \in \Gamma} m^n(\gamma, \theta)  
&\le&  \sum_{\theta \in \Gamma} |\theta| m^n(\gamma, \theta)\nonumber\\
&=&|\gamma|\sum_{\gamma_1 \not\sim \gamma} \frac{|\gamma_1|}{|\gamma|}w(\gamma_1) \, \sum_{\gamma_2
\not\sim \gamma_1} \frac{|\gamma_2|}{|\gamma_1|} w(\gamma_2) \cdots
\sum_{\theta \not\sim \gamma_{n-1}}
\frac{|\theta|}{|\gamma_{n-1}|}w(\theta)\nonumber\\
&\le&  |\gamma|\Bigl(\sum_{\theta \ni \vec0} |\theta|
w(\theta)\Bigr)^n = |\gamma|\beta^n,\label{bb1}
\end{eqnarray}
where the inequality in \eqref{bb1} follows from
\[
\sum_{\gamma':\gamma' \not\sim \gamma} |\gamma'|w(\gamma')\le |\gamma|  \sum_{\gamma'\ni\vec0}|\gamma'|w(\gamma'). \qedhere
\]
\end{proof}

\begin{corollary}
  \label{a*i} 
If $\alpha>\alpha^*(\uU)$, as defined in \eqref{alpha*}, then the clan of
ancestors of any point $(\gamma,t,s)$ is finite for almost all realizations of $\cN$.
\end{corollary}
\begin{proof}
  By Lemma  \ref{bb9} the size of the clan of ancestors is dominated by the total population of the
  branching process. This population is finite if
  $\beta(\alpha)<1$ by Lemma \ref{beta-alpha}.
\end{proof}

\section{Proofs of the main theorems}
\label{s4}
\paragraph{Proof of Theorem \ref{t1}.}\

Let $\eta^\Lambda_t$ and $\eta_t=\eta^{\Z^d}_t$ be the stationary processes defined in \eqref{el} and \eqref{def.eta}. 
Fix an arbitrary time $t\in\R$ and consider the $t$ marginal of the process $\eta^\Lambda_t$. Define $\zeta_\Lambda\in S_\Lambda$ as the permutation with cycles indicated by $\eta^\Lambda_t\in\{0,1\}^{\Gamma_\Lambda}$, for each $\Lambda\subseteq \Z^d$. That is, $\gamma\in\zeta_\Lambda$ if and only if $\eta_t^\Lambda(\gamma)=1$.  

(i)  For finite $\Lambda$ the marginal distribution of the process $\eta^\Lambda_t$ at each fixed $t\in\R$ is $G_\Lambda$, as discussed following \eqref{el}.  

(ii) Thermodynamic limit: 
  Under the condition $\alpha>\alpha^*(\uU)$, Corollary \ref{a*i} implies that the clan of ancestors of any
  point in $\cN$ is finite with probability one. 
Then Theorem \ref{tl} implies the almost sure thermodynamic limit $\zeta_\Lambda(x)\to\zeta_{\Z^d}(x)$, as $\Lambda\nearrow\Z^d$, and hence $G_\Lambda\to\mu$, the law of $\eta^{\Z^d}_t$.

(iii)  By item (ii) $\mu$ is a weak limit of specifications, hence a Gibbs measure.  Since $\eta^{\Z^d}_t$ is a space-time translation invariant function of the Poisson process $\cN$,  the ergodicity of $\cN$ implies the spatial ergodicity of the law of $\eta^{\Z^d}_t$.  

(iv) Uniqueness follows from Theorem \ref{uniqueness}. \qed

 \paragraph{$\tau_v$ boundary conditions. Proof of Theorem \ref{t2}. }

Fix $v\in \Z^d$ and $\alpha>\alpha^*_v(\uU)$ and define the potential $V_v$ as in \eqref{vv}. 
Then $\uU_v(0)=0$ and $\uU_v$ is strictly convex because $\uU$ is. Also, $\alpha^*(\uU_v) =  \alpha^*_v(\uU) <\infty$ by hypothesis.
Hence $\uU_v$ satisfies the hypothesis of Theorem \ref{t1} and there exist a process $(\zeta_{\Lambda,v},\Lambda\subseteq\Z^d)$,
$\zeta_{\Lambda,v}\in S_\Lambda$, such that (i) to (iv) of that
theorem hold. 

(i) If $\Lambda$ is finite and $\sigma\in S_\Lambda$, 
\begin{align}
  P(\tau_v\zeta_{\Lambda,v}=\tau_v\sigma) &=P(\zeta_{\Lambda,v}=\sigma)\notag\\
&= \frac1{Z_{\Lambda,v}} \prod_{\gamma\in\sigma}
  \exp \Bigl\{\sum_{x\in\{\gamma\}} \uU_v(\gamma(x)-x)\Bigr\}\notag\\[2mm]
&= \frac1{Z_{\Lambda,v}} \prod_{\gamma\in\sigma}
  \exp \Bigl\{\sum_{x\in\{\gamma\}}
  \big[\uU(\gamma(x)+v-x)-\uU(v)\big]\Bigr\}\notag\\
&=G_{\Lambda|\tau_v}(\tau_v\sigma).
\end{align}
The remaining items follow from (i) and the statements (ii) to (iv) of Theorem \ref{t1}. \qed

\paragraph{The Gaussian potential. Proof of Theorem \ref{t3}. }

Assume $\uU$ is the Gaussian potential $\uU(x)=\|x\|^2$ in $\Z^d$. Let $\rho_0$ as in \eqref{rho-c} \, and $\alpha>0$ such that 
$\sum_{x\in\Z^d\setminus\{\vec0\}}\exp(-\alpha\|x\|^2)<\rho_0$. By Lemma \ref{rho} below this implies $\alpha^*(\|\cdot\|^2)<\infty$. We compute the explicit bound for $\alpha^*(\|\cdot\|^2)$ later in \eqref{explicit}. 

Fix $v\in \Z^d$.

(i)  Any permutation $\xi \in S_{\Lambda|\tau_v}$ is a finite perturbation of $\tau_v$ and by Lemma \ref{tau'}, $\tau_{-v}\xi=\gamma_1\dots\gamma_n$, a composition of disjoint finite cycles in $S_\Lambda$. We can then write
  \begin{align}
  G_{\Lambda|\tau_v}(\xi) &= \frac1{Z_{v,\Lambda}}\exp\Bigl\{-\alpha\sum_{\gamma\in\tau_{-v}\xi}\sum_{x\in\{\gamma\}} (\|\gamma(x)+v-x\|^2 -\|v\|^2)-\alpha\sum_{x\in\Lambda}\|v\|^2    \Bigr\}  \\
&= \frac1{Z'_{v,\Lambda}}\prod_{\gamma\in \tau_{-v}\xi} w_v(\gamma),
  \end{align}
where $Z'_{v,\Lambda}= Z_{v,\Lambda} \exp\{\alpha\sum_{x\in\Lambda} \|v\|^2\}$ and $w_v$ is defined in \eqref{weight}. If $\gamma$ is a cycle with support in $\Lambda$
\begin{align}
  \sum_{x\in\{\gamma\}} \big(\|\gamma(x)+v-x)\|^2-\|v\|^2\big)
&=\sum_{x\in\{\gamma\}} \|x-\gamma(x)\|^2 +2v \cdot \sum_{x\in\{\gamma\}} (x-\gamma(x)) \notag\\
&= \sum_{x\in\{\gamma\}} \|x-\gamma(x)\|^2
\end{align}
as $\sum_{x\in \{\gamma\}} x-\gamma(x)=0$. 
This implies $w_v(\gamma)=w(\gamma)$ defined in \eqref{weight},  and
$  G_{\Lambda|\tau_v}(\xi) = G_\Lambda(\tau_{-v}\xi)$ for $\xi\in S_{\Lambda|\tau_v}$. But this is equivalent to $G_{\Lambda|\tau_v}=G_\Lambda\tau_v$.

(ii) Take $v\in\Z^d$ and $\alpha>\alpha^*(\uU)$ and let $\zeta_\Lambda$ be as constructed in Theorem \ref{t1}. Since $\zeta_\Lambda$ has law
$G_\Lambda$, by (i), $ \tau_v\zeta_\Lambda$ has distribution
$G_{\Lambda|\tau_v}$.  The almost sure thermodynamic limit of item (ii), Theorem~\ref{t1} implies
\[
\lim_{\Lambda\nearrow\Z^d}\tau_v\zeta_\Lambda(x)= \tau_v\zeta_{\Z^d}(x)
\quad {\rm a.s.}, \quad\hbox{for all } x \in \Z^d.
\]
Now  $\zeta_{\Z^d}$ is distributed according to  $\mu$, then
\[
\lim_{\Lambda\nearrow\Z^d}G_{\Lambda|\tau_v} = \mu\tau_v 
\quad\hbox{weakly}. 
\]

(iii) The ergodicity of $\mu\tau_v$ follows from the ergodicity of $\mu$ proved in Theorem~\ref{t1}~(iii). \qed

\section{Bounds on $\alpha^*(\uU)$}
\label{s5}


\subsection{A general bound}
We start with a general bound. Following \cite{MR2360227}, define
\begin{equation}
  \label{rho-alpha}
  \rho(\uU,\alpha) := \sum_{x\in\Z^d\setminus\{\vec0\}} e^{-\alpha \uU(x)}\,.
\end{equation}
The proof of the following lemma is taken from the proof of Theorem 2.1  in
\cite{MR2360227}.  

\begin{lemma}
  \label{rho}
Call $\beta=\beta(\uU,\alpha)$ and let $\rho = \rho(\uU,\alpha)$. Then, 
\begin{equation}
  \label{bb7}
  \beta\;\le\;\frac{\rho}{\big[1-\rho\big]^2}-\rho.
\end{equation}
\end{lemma}

\begin{proof}
  Compute 
  \begin{eqnarray}
\label{bb4}
  \beta&=& \sum_{\theta \ni\vec0}|\theta| w(\theta) 
\;=\; \sum_{n\ge 2} n \sum_{\theta \ni\vec0;\,|\theta|=n} w(\theta)\,.
\end{eqnarray}
The second sum, indexed by $\theta$, on the right of \eqref{bb4} can be re-written as
\begin{equation}
\sum_{x_1,\dots,\,x_{n+1}\in\Z^d}\;\one\big\{x_1=x_{n+1}=0;\,x_i\neq
x_j\,,\;i,j\in\{1,\dots,n\}\big\}\; 
\prod_{i=1}^{n}e^{-\alpha \uU(x_{i+1}-x_{i})}\,.   \label{a66}
\end{equation}
Dominate the indicator function in \eqref{a66} by 
$\one\big\{x_1=0;\;x_i\neq x_{i+1}\,,\; i\in\{1,\dots,n\}\big\}$, to dominate
\eqref{a66}  by
\[
\sum_{y_1,\dots,\,y_n\in\Z^d\setminus\{\vec0\}}\; \prod_{i=1}^{n}e^{-\alpha \uU(y_{i})}
\;=\; \Bigl(\sum_{x\in\Z^d\setminus\{\vec0\}}  e^{-\alpha \uU(x)}\Bigr)^n=\rho^n\,.
\]
We conclude that $\beta\le \sum_{n\ge 2}n\rho^n$ which is equivalent to \eqref{bb7}.
\end{proof}
Let
\begin{equation}
  \label{rho-c}
\rho_0 \hbox{ be  the unique solution $r\in[0,1]$ to }\;\frac{r}{(1-r)^2}-r=1.
\end{equation}
Solving the equation one gets $\rho_0\approx 0.44504$.

\begin{corollary}
\label{8888}
  If $\rho(\uU,\alpha)<\rho_0$ then  $\beta(\uU,\alpha)<1$. In particular,
\begin{equation}
  \label{alphac}
\alpha^*(\uU )\le \inf\{\alpha>0: \rho(\uU,\alpha)\le \rho_0\}\,.
\end{equation}

\end{corollary}

\subsection{Examples} 

\paragraph{The Gaussian potential} 
In this case
\begin{align}
\rho(\|\cdot\|^2,\alpha) 
&=  \sum_{z\in\Z^d\setminus\{\vec0\}}e^{-\alpha \|z\|^2}
= \Bigl(\sum_{k\in\Z}e^{-\alpha k^2}\Bigr)^d-1  \notag\\
&
\le \Bigl(1+\int_{-\infty}^\infty e^{-\alpha x^2}dx\Bigr)^d-1 = \Bigl(1+\sqrt{\pi/\alpha}\Bigr)^d-1,
\end{align}
which implies the following explicit bound for $\alpha^*$:
\begin{align}
\label{explicit}
  \alpha^*(\|\cdot\|^2)\le \pi\big((\rho_0+1)^{1/d}-1\big)^{-2}\approx \pi\big(1.44504^{1/d}-1\big)^{-2}.
\end{align}
For $d=2$ this gives $\alpha^*\le 76.9176$; for $d=3$, $\alpha^*\le 184.305\,$. 

\paragraph{Differentiable, strictly convex potentials} Let $\uU: \R^d\to \R^+\cup\{+\infty\} $ be a potential such that for each $v\in \Z^d$ there exists a constant $m(v)>0$ satisfying 
\begin{align}
\label{m1v}
\uU(y)\ge \uU(x)+\nabla \uU(x)^T\cdot(y-x)+m(v)\|y-x\|^2\,\qquad \mbox{for any } x,\,y \in \R^d.
\end{align}
Then, for any cycle $\gamma \in  \Gamma$,
\begin{align}
\label{m2v}
\sum_{x\in \gamma}\uU(\gamma(x)+v-x)-\uU(v)\ge m(v)\sum_{x\in \gamma}\|\gamma(x)-x\|^2,
\end{align}
and $\alpha^*_v(\uU)\le \frac{1}{m(v)} \,\alpha^*(\|\cdot\|^2)$.
 
In particular, if $V$ is a strongly convex potential then \eqref{m1v} holds with a constant $m$ uniformly in $v\in \Z^d$, and 
$\alpha^*_v(\uU)\le \frac{1}{m} \,\alpha^*(\|\cdot\|^2)$. For instance, in $1$-dimension, $x^2$ and $e^{x^2}$ are strongly convex potentials.

\paragraph{Polynomial potentials} Let $\uU:\R^d\to \R^+$ be a strictly convex polynomial, $\uU(\vec0)=0$, with a positive definite Hessian at all points. 
Given $v\in \Z^d$, there exists $b(v)>0$
such that
\begin{equation}
\label{large}
\big[\uU(v+y)-\uU(v)-\nabla \uU(v)\cdot y\big] \,\one_{\|y\|\ge b(v)} \ge \frac{1}{2}\, \uU(y)\, \one_{\|y\|\ge b(v)}
\end{equation}
Let now $\gamma \in \Gamma$, and write
\begin{align*}
&\sum_{x\in \{\gamma\}} \uU(\gamma(x)-x+v)-\uU(v)\\
&\hspace{.1cm}=\sum_{x\in \{\gamma\}} \uU(\gamma(x)-x+v)-\uU(v)-\nabla \uU(v) \cdot (\gamma(x)-x)=I_1+I_2
\end{align*}
with
\begin{align*}
&I_1=\sum_{x\in \{\gamma\},\,\|x-\gamma(x)\|<b(v)} \uU(\gamma(x)-x+v)-\uU(v)-\nabla \uU(v) \cdot (\gamma(x)-x)\\
&I_2=\sum_{x\in \{\gamma\},\,\|x-\gamma(x)\|\ge b(v)} \uU(\gamma(x)-x+v)-\uU(v)-\nabla \uU(v) \cdot (\gamma(x)-x)
\end{align*}
By \eqref{large}
\begin{equation}
\label{cota1}
I_2\ge \,\frac{1}{2} \sum_{x\in \{\gamma\},\,\|x-\gamma(x)\|\ge b(v)} \uU(\gamma(x)-x).
\end{equation}
On the other hand, since the set $\{y\in \Z^d, \|y-v\|<b(v)\}$ is finite and the Hessian $H\uU=\big(\frac{\partial ^2 V}{\partial x_i \partial x_j}\big)_{1\le i,j\le d}$ is positive definite at all points, there exists $m>0$ such that
\[
\uU(\gamma(x)-x+v)-\uU(v)-\nabla \uU(v) \cdot (\gamma(x)-x)\ge m\, \|\gamma(x)-x\|^2
\] 
for all $\|\gamma(x)-x\|<b(v)$. As a result,
\begin{equation}
\label{cota2}
I_1\ge \,m \sum_{x\in \{\gamma\},\,\|x-\gamma(x)\|<b(v)} \|\gamma(x)-x\|^2.
\end{equation}
Finally, for any integer neighborhood of the origin that excludes the origin itself there exists another constant $m'$ such that $m' \uU(x)\le \|x\|^2$. Together with (\ref{cota1}, \ref{cota2}), we obtain 
\begin{equation*}
\sum_{x\in \{\gamma\}} \uU(\gamma(x)-x+v)-\uU(v)\ge C(v) \sum_{x\in \{\gamma\}} \uU(\gamma(x)-x)
\end{equation*}
for some constant $C(v)>0$, and $\alpha^*_v(\uU)\ge \,\frac{1}{C(v)}\,
\alpha^*(\uU)$, where $\alpha^*(\uU)$ can be bounded as in Corollary \ref{8888}. 

Finally, note that some naturally arising polynomial potentials such as $\uU(x)=x^4$ fail to have a positive definite Hessian at all points. In this case the above argument still applies, provided the set of points where the Hessian is not positive definite does not affect the computation leading to \eqref{cota2}. In other words, one just needs to check that the 
Hessian appearing in the remainder term of the 1st degree Taylor expansion of
$\uU(z)$ around $v$ is positive definite, for all (finitely many) integer points
$z$ in the neighborhood $\|z-v\|<b(v)$. For instance, in the case of
$\uU(x)=x^4$, or its $d$-dimensional version $\uU(x)=\|x\|^4$, the Hessian fails to be
 positive definite only at the origin, and the argument works fine.

\paragraph{Finite-range potentials}
We say that a potential $\uU$ is \emph{finite-range} if $\uU(x)=\infty$ for all but a finite number of $x$'s. 
 Consider for instance the nearest neighbor potential $\uU:\Z^2\to\{0,1,\infty\}$ defined by $\uU(0)=0$, $\uU(x)=1$ if $\|x\|=1$ and $\uU(x)=\infty$ if $\|x\|>1$. In $d=2$ the specifications for the permutations are very similar to the ones for the Peierls contours of the Ising model. There are some differences between these models: (a)  while contours may have self intersections, cycles are not allowed to; (b) two-point cycles do not determine a contour, and (c) given the set of at least 3 sites in the support of a cycle, there are two possible ways of going through them, clockwise and counter clockwise. Other than these observations,  the approach works exactly as in the contour case studied in \cite{MR1849182}.

\section*{Acknowledgments}
We are grateful to both referees for several comments that helped improve the paper. I.A.\/ would like to thank Stefan Grosskinsky and Daniel Ueltschi for
many fruitful discussions as well as for their warm welcome to the University
of Warwick.  This research has been supported by the grant PICT
2012-2744 ``Stochastic Processes and Statistical Mechanics", the
project UBACyT 2013-2016 20020120100151BA and the MathAmSud project 777/2011 ``Stochastic Structure of Large Interactive Systems". F.L.\/ is
partially supported by the CNPq-Brazil fellowship 304836/2012-5 and a L'Or\'eal
Fellowship for Women in Science. This article was produced as part of the activities of FAPESP  Research, Innovation and Dissemination Center for Neuromathematics, grant
2013/07699-0, S\~ao Paulo Research Foundation.

\bibliography{afgl.bib}

\bibliographystyle{plain}

\end{document}